\documentclass[reqno]{amsart}
\usepackage{amsmath, amssymb, amsthm, comment, graphics, graphicx, hyperref, url, longtable, stmaryrd}
\newcommand{\abs}[1]{\lvert#1\rvert}
\newcommand{\gfextn}[0]{jpeg}

\newtheorem{theorem}{Theorem}[section]
\newtheorem{lemma}[theorem]{Lemma}

\newtheorem*{remark}{Remark}
\numberwithin{equation}{section}
\newcommand\TS{\rule{0pt}{2.6ex}} 
\newcommand\BS{\rule[-1.2ex]{0pt}{0pt}} 

\begin{document}
\title{Explicit zero-free regions for the Riemann zeta-function}

\author[M.~J. Mossinghoff]{Michael J. Mossinghoff}
\address{Center for Communications Research\\
Princeton, NJ, USA}
\email{m.mossinghoff@idaccr.org}

\author[T.~S. Trudgian]{Timothy S. Trudgian}
\address{School of Science\\ The University of New South Wales Canberra, Australia}
\thanks{The second author was partially supported by Australian Research Council Future Fellowship FT160100094.}
\email{t.trudgian@adfa.edu.au}

\author[A. Yang]{Andrew Yang}
\address{School of Science\\ The University of New South Wales Canberra, Australia}
\email{andrew.yang1@adfa.edu.au}

\date\today

\begin{abstract}
We prove that the Riemann zeta-function $\zeta(\sigma + it)$ has no zeros in the region $\sigma \geq 1 - 1/(55.241(\log\abs{t})^{2/3} (\log\log \abs{t})^{1/3})$ for $\abs{t}\geq 3$.
In addition, we improve the constant in the classical zero-free region, showing that the zeta-function has no zeros in the region $\sigma \geq 1 - 1/(5.558691\log\abs{t})$ for $\abs{t}\geq 2$.
We also provide new bounds that are useful for intermediate values of $\abs{t}$.
Combined, our results improve the largest known zero-free region within the critical strip for $3\cdot10^{12} \leq \abs{t}\le \exp(64.1)$ and $\abs{t} \geq \exp(1000)$.
\end{abstract}

\maketitle

\section{Introduction}

Let $\zeta(s)$ denote the Riemann zeta-function, where $s= \sigma+ it$ is a complex variable.
All non-trivial zeros of $\zeta(s)$ lie in the critical strip with $0<\sigma <1$. Determining regions in the critical strip that are devoid of zeros of $\zeta(s)$ is of great interest in number theory. Such regions take the shape $\sigma \geq 1 - 1/f(\abs{t})$ for some function $f(t)$ tending to infinity with $t$. The so-called classical zero-free region has $f(t) = R_0 \log t$, where $R_0$ is a positive constant. An asymptotically larger region, proved by Korobov \cite{Korobov} and Vinogradov \cite{Vinogradov}, has $f(t) = R_1 (\log t)^{2/3} (\log\log t)^{1/3}$, for some constant $R_1 > 0$.

Considerable effort has been made to make these results explicit, and we briefly recall the sharpest zero-free regions known.
For the Korobov--Vinogradov region, in 2000 Ford \cite{Ford2000} (see also \cite{Ford2} for minor corrections) established two explicit bounds, one holding for essentially all $t$, and a larger one holding for sufficiently large $t$. 
In the former case, Ford proved that there are no zeros of $\zeta(s)$ in the region
\begin{equation}\label{eqnFordKV}
\sigma \geq 1 -\frac{1}{57.54 (\log \abs{t} )^{2/3} (\log\log \abs{t})^{1/3}}, \quad \abs{t}\geq 3.
\end{equation}
In the latter case, Ford showed that there are no zeros when $t$ is sufficiently large and
\begin{equation}\label{eqnFordKVA}
\sigma \geq 1 -\frac{1}{49.13 (\log \abs{t} )^{2/3} (\log\log \abs{t})^{1/3}}.
\end{equation}
This was recently slightly improved by Nielsen \cite{Nielsen}, who showed that there are no zeros with
\begin{equation}\label{eqnNielsen}
\sigma \geq 1 -\frac{1}{49.08 (\log \abs{t} )^{2/3} (\log\log \abs{t})^{1/3}}
\end{equation}
for sufficiently large $\abs{t}$.
For the classical region, in 2015 the first two authors \cite{HoffTrudgian} proved that the region
\begin{equation}\label{eqnMTc}
\sigma \geq 1 -\frac{1}{5.573412 \log \abs{t}}, \quad \abs{t}\geq 2
\end{equation}
is devoid of zeros of the zeta-function. This region is wider than that of \eqref{eqnFordKV} for all $\abs{t} \leq \exp(10151.5)$.

We improve \eqref{eqnFordKV}, \eqref{eqnNielsen} and \eqref{eqnMTc} in this article.
For the first of these, we establish the following theorem.

\begin{theorem}\label{Main}
There are no zeros of $\zeta(\sigma + it)$ for $\abs{t}\geq 3$ and 
\begin{equation}\label{fun}
\sigma \ge 1 - \frac{1}{55.241(\log\abs{t})^{2/3} (\log\log \abs{t})^{1/3}}.
\end{equation}
\end{theorem}

Our improvement in Theorem~\ref{Main} over \eqref{eqnFordKV} is a result of several ingredients.
\begin{itemize}
\item We employ a trigonometric polynomial of large degree: this produces our largest improvement.
\item We optimize the choice of certain parameters to account for secondary error terms. 
\item We use improved intermediate zero-free regions to cover medium-sized $t$ values where the argument for the asymptotic region does not perform as well.
\item We employ sharper explicit estimates on the growth of the zeta-function in the critical strip.
\item We use a new height up to which the Riemann hypothesis (RH) has been proved.
\end{itemize}

It appears substantially more difficult to improve the constant in the asymptotic region \eqref{eqnNielsen}.
We identify three potential avenues of improvement: using a better smoothing function, finding a more favorable trigonometric polynomial, and improving the constant $B$ in Richert's bound---see \eqref{rick}.
In \cite{heath_brown_zero_1992} and \cite{Xylouris}, improvements related to the first approach have been largely explored.
In this work, we exploit the second method to produce the following improvement.  

\begin{theorem}\label{Main2}
For sufficiently large $\abs{t}$, there are no zeros of $\zeta(\sigma + it)$ with
\begin{equation}\label{Theorem2}
\sigma \ge 1 - \frac{1}{48.1588(\log \abs{t})^{2/3}(\log \log \abs{t})^{1/3}}.
\end{equation}
\end{theorem}

We remark that Nielsen's result \eqref{eqnNielsen} improved \eqref{eqnFordKVA} by replacing a particular trigonometric polynomial that was used in Ford's argument.
Ford used a polynomial of degree $4$, while Nielsen adopted one of degree $5$ (see Section~\ref{subsecTrigPolys}).
We use polynomials of substantially larger degree here for our improvements.
Also, we mention that Khale \cite{Khale} recently established analogues of Theorems~\ref{Main} and~\ref{Main2} for Dirichlet $L$-functions. While some of the ingredients we employ in this article do not appear to be suitable for use with Dirichlet $L$-functions, we note that the use of higher-degree trigonometric polynomials may lead to some improvements in \cite{Khale}.

For the classical region, in \cite{HoffTrudgian} the inequality \eqref{eqnMTc} was established by extending some work of Kadiri \cite{Kadiri} by constructing a more favorable trigonometric polynomial, by optimizing some analytic arguments, and by employing the verification of RH up to $3.06\cdot10^{10}$ from \cite{Platt15}.
It was also recorded in \cite{HoffTrudgian} that if RH were verified for $\abs{t} \le 3\cdot 10^{11}$, then the constant in \eqref{eqnClassical} could be replaced by $5.5666305$, and this is now permissible due to the verification performed in \cite{platt_riemann_2021}.
We take the opportunity here to record a further improvement for the classical region, using two ideas.
First, RH was verified for $\abs{t} \le 3\cdot 10^{12}$ in \cite{platt_riemann_2021}.
Second, in 2014, Jang and Kwon \cite{jang_note_2014} derived another improvement to Kadiri's result by different means, involving the replacement of a particular smoothing function employed in \cite{Kadiri}, and we incorporate this as well.
We prove the following theorem.

\begin{theorem}\label{thmClassical}
There are no zeros of $\zeta(\sigma + it)$ for $\abs{t}\geq 2$ and 
\begin{equation}\label{eqnClassical}
\sigma \geq 1 - \frac{1}{5.558691 \log \abs{t}}.
\end{equation}
\end{theorem}

We remark that the region \eqref{eqnClassical} is wider than that of \eqref{fun} for $\abs{t}<\exp(8928)$.

In order to obtain a good constant in Theorem~\ref{Main}, it is necessary to deduce another zero-free region to cover medium-sized values of $t$, in addition to \eqref{eqnClassical}.
We obtain such a zero-free region by using bounds on the growth rate of the zeta-function on the half-line and establish the following theorem.

\begin{theorem}[Intermediate zero-free region]\label{thmIntZFR}
There are no zeros of $\zeta(\sigma + it)$ for $\abs{t} \ge \exp(1000)$ and 
\begin{equation*}
\sigma > 1 - \frac{0.05035}{h(t)} + \frac{0.0349}{h^2(t)},
\end{equation*}
where $h(t) = \frac{27}{164}\log \abs{t} + 7.096$.
\end{theorem}

We record one additional result of Ford: from \cite[Thm.\ 3]{Ford2000}, we see that if $t\geq 1.88\cdot 10^{14}$ then $\zeta(\sigma+it)\neq0$ for
\begin{equation}\label{eqnFordMedium}
\sigma \geq 1 - \frac{0.04962-\frac{0.0196}{J(t)+1.15}}{J(t) + 0.685 + 0.155\log\log t},
\end{equation}
where $J(t) = \frac{1}{6}\log t +\log\log t +\log3$.

By combining these results, we may summarize the largest known zero-free region for the Riemann zeta-function within the critical strip for each height $t$.
For $\abs{t}\leq3\cdot10^{12}$, all zeros are known to lie on the critical line.
The region of Theorem~\ref{thmClassical} provides the best known bound for $3\cdot10^{12}<\abs{t}\leq\exp(64.1)$, then for $\exp(64.1)<\abs{t}\leq\exp(1000)$ the expression \eqref{eqnFordMedium} produces the widest region.
After this, for $\exp(1000)<\abs{t}\leq\exp(52238)$, Theorem~\ref{thmIntZFR} is best, and for $\abs{t}>\exp(52238)$ the result of Theorem~\ref{Main} produces the widest known region.

This article is organized in the following manner.
Section~\ref{secApplications} lists some immediate applications of our results.
Section~\ref{secPrelims} reviews some results we require from the literature.
Section~\ref{sec:aux_lemmas} contains a number of lemmas that we require for the proofs of our main theorems.
Section~\ref{sec:mainproof} contains the proof of Theorem~\ref{Main}.
Section~\ref{sec:intermediate_proof} describes the proof of Theorem~\ref{thmIntZFR}, and Section~\ref{sec:proofMain2} has the proof of Theorem~\ref{Main2}.
Section~\ref{sec:comp} summarizes our computations for determining the constants in Theorems~\ref{Main} and~\ref{Main2}.
Section~\ref{sec:Classical} describes the improvement in the classical region for Theorem~\ref{thmClassical}, and Section~\ref{secFuture} suggests some potential future work.

\section{Some immediate applications}\label{secApplications}

One application of Theorem~\ref{Main} concerns the existence of primes between successive cubes. Dudek \cite{DudekCubes} used \eqref{eqnFordKV} to show that there exists a prime between $n^3$ and $(n + 1)^3$ for all $n \ge \exp(\exp(33.3))$. This was subsequently improved by Cully-Hugill \cite{cully_hugill_primes} to all $n \ge \exp(\exp(32.892))$. By substituting Theorem~\ref{Main} into \cite[\S4]{cully_hugill_primes} and choosing $\alpha = 1/3 + 10^{-13}$, we obtain the modest improvement $n \ge \exp(\exp(32.76))$.

Another application is explicit bounds on the error term of the prime number theorem---see Johnston and Yang \cite{johnston_some_2022}, who developed work of Platt and Trudgian \cite{PT} particularly with reference to the Korobov--Vinogradov zero-free region.
Other related recent works such as \cite{BMOR,BKLNW,CHL,Fiori} required explicit zero-free regions to obtain precise estimates.

As per Ford \cite[p.\ 566]{Ford}, one can use Theorem~\ref{Main2} to improve on the error term in the prime number theorem. From the work of Pintz \cite{Pintz}, a zero-free region like that in \eqref{Theorem2} with $c$ in place of $48.1588$ shows that  
\begin{equation*}
\pi(x) - \textrm{li}(x) \ll x \exp\{ -d(\log x)^{3/5} ( \log\log x)^{-1/5}\},
\end{equation*}
where 
\begin{equation*}
d = \left( \frac{5^{6}}{2^{2} \cdot 3^{4} \cdot c^{3}}\right)^{1/5}.
\end{equation*}
With $c = 48.1588$ from Theorem~\ref{Main2}, we obtain $d \ge 0.2123$, improving on the current best value of $0.2098$.

\section{Preliminaries}\label{secPrelims}

To assist in our argument we review the following results from the literature. 
\subsection{Bounds on $\zeta(s)$ near the 1-line}
The underlying philosophy 
in furnishing zero-free regions is to estimate $\zeta(s)$ in a region close to the line $\sigma = 1$. Richert's theorem \cite{Richert67} (see also \cite[p.\ 135]{Titchmarsh}), accomplishes this by proving that there exist positive constants $A$ and $B$ for which
\begin{equation}\label{rick}
\abs{\zeta(\sigma + it)} \leq A \abs{t}^{B(1-\sigma)^{3/2}} (\log \abs{t})^{2/3}, \quad \abs{t}\geq 3,\; \frac{1}{2} \leq \sigma \leq 1.
\end{equation}
Ford \cite{Ford} obtains a relatively small value of $B$ while maintaining a completely explicit value of $A$: in (\ref{rick}) one can take $A= 76.2$ and $B=4.45$. We remark in passing that the advances made in \cite{Preo} may be applied to Ford's paper \cite{Ford}, and should lead to a slight reduction in the value of $B$ in (\ref{rick}). 

We also require bounds on $\zeta(s)$ slightly to the right of $\sigma = 1$.
A result recorded in Bastien and Rogalski \cite{Bastien_convexite_2002}, originally due to O. Ramar\'e, states that for $\sigma > 1$, we have
\begin{equation}\label{sigma_bound_1}
\zeta(\sigma) \le \frac{e^{\gamma(\sigma - 1)}}{\sigma - 1},
\end{equation}
where $\gamma$ is Euler's constant. Using the Maclaurin series for $e^{x}$ one sees that the approximation in (\ref{sigma_bound_1}) is very good.

\subsection{Bounds on $\zeta(s)$ on the half-line}
We recall the recent sub-Weyl bound due to Patel \cite{patel_explicit_2021}: for $\abs{t} \ge 3$, we have
\begin{equation}\label{subweyl_bound}
\abs{\zeta(\tfrac{1}{2} + it)} \leq 307.098\abs{t}^{27/164}.
\end{equation}
This is required in the proof of Theorem~\ref{thmIntZFR}  in Section~\ref{sec:intermediate_proof}.

\subsection{Trigonometric polynomials}\label{subsecTrigPolys}
For any $K \ge 2$, let 
\begin{equation*}
P_K(x) = \sum_{k=0}^{K} b_{k} \cos(kx)
\end{equation*}
where $b_k$ are constants such that $b_{k}\geq 0$, $b_{1}>b_{0}$ and $P_K(x) \geq 0$ for all real $x$. We refer to $P_K(x)$ as a $K$th degree non-negative trigonometric polynomial. Choosing a favorable $P_K$ plays an integral part in determining the size of the zero-free region: see \cite{HoffTrudgian} for a detailed history of this problem. 

In \cite{Ford2000}, Ford employed the degree $4$ polynomial
\begin{equation}\label{ford_poly_coefficients}
P_4(x) = (0.225 + \cos x)^2(0.9 + \cos x)^2,
\end{equation}
while Nielsen \cite{Nielsen} adopted the degree $5$ function
\begin{equation}\label{eqnNelsenPoly}
P_5(x) = (1+\cos x)(0.1974476\ldots + \cos x)^2(0.8652559\ldots + \cos x)^2.
\end{equation}
We employ simulated annealing in a large-scale search to determine favorable trigonometric polynomials of higher degree.
We utilize two polynomials in this article.
Theorem~\ref{Main} relies on a polynomial $P_{40}(x)$ with degree $40$ having
\begin{equation}\label{eqnPoly40}
b_0 = 1,\quad b_1 = 1.74600190914994,\quad b = \sum_{k = 1}^{40}b_k = 3.56453965437134,
\end{equation}
while Theorem~\ref{Main2} employs a polynomial $P_{46}(x)$ with degree $46$ where
\begin{equation}\label{eqnPoly46}
b_0 = 1,\quad b_1 = 1.74708744081848,\quad b = \sum_{k = 1}^{46}b_k = 3.57440943022073.
\end{equation}
We show the full polynomials in Tables~\ref{tableP46} and~\ref{tableP40} in Section~\ref{sec:comp}, and provide details there on the process used to find these polynomials as well as a justification that they are non-negative. 

\subsection{Estimates of the zero-counting function}
We also require bounds on the error term of estimates of $N(T)$, the number of zeros of $\zeta(s)$ with $0 \le t < T$. For large $T$, the current best known explicit result is due to Hasanalizade, Shen and Wong \cite[Cor.\ 1.2]{hasanalizade_counting_2021}, who proved that
\begin{equation}\label{NT_bound}
\left|N(T) - \frac{T}{2\pi}\log \frac{T}{2\pi e}\right| \le 0.1038\log T + 0.2573\log\log T + 9.3675
\end{equation}
for all $T \ge e$. The bound in (\ref{NT_bound}) can be improved in the region where RH is known to hold \cite{platt_riemann_2021}, and for some intermediate values of $T$ as well using \cite{PT2}.
Nevertheless, the bound in \eqref{NT_bound} is the best available at present for $T\geq 10^{410}$.
While that bound could be reduced using \eqref{subweyl_bound}, it suffices for our purposes.

\section{Required lemmas}\label{sec:aux_lemmas}

Before proceeding to the proof of Theorem~\ref{Main} in the next section, we review some useful lemmas.
Throughout, let 
\begin{equation*}
L_1 = L_1(t) = \log (Kt + 1),\quad L_2 = L_2(t) = \log\log(Kt + 1),
\end{equation*}
where $K$ is the degree of the trigonometric polynomial under consideration.
Let $b_0$, $b_1$ and $b$ be as in \eqref{eqnPoly40} or \eqref{eqnPoly46}.
In addition, assume throughout that \eqref{rick} holds for some $A, B > 0$.
Unless otherwise stated, our results will remain valid for all $A, B > 0$ for which \eqref{rick} holds.
Since $\zeta(\overline{s}) = \overline{\zeta(s)}$, it suffices to consider only $t > 0$ throughout.

As in Ford \cite{Ford2000}, the main tool used to pass from upper bounds on $\zeta(s)$ to a zero-free region is the ``zero-detector", which expresses $-\Re\frac{\zeta'}{\zeta}(s)$ as an integral involving $\zeta(s)$ over two vertical lines on either side of $\sigma = 1$ (plus a small error term). Concretely, we have the following lemma.

\begin{lemma}[Zero-detector for $\zeta$]\label{fordlem22}
Let $s = 1 + it$ with $t \ne 0$ and $\rho$ run through the non-trivial zeros of $\zeta(s)$. For all $\eta > 0$, except for a set of Lebesgue measure $0$, we have 
\begin{equation*}
\begin{split}
-\Re \frac{\zeta'}{\zeta}(s) &= \frac{1}{4\eta}\int_{-\infty}^{\infty}\frac{\log\abs{\zeta(s - \eta + \frac{2\eta iu}{\pi})} - \log\abs{\zeta(s + \eta + \frac{2\eta iu}{\pi})}}{\cosh^2 u}\,\text{d}u\\
&\qquad\qquad + \frac{\pi}{2\eta}\sum_{\abs{\Re(s - \rho)} \le \eta}\Re \cot\left(\frac{\pi(\rho - s)}{2\eta}\right),
\end{split}
\end{equation*}
where in the sum, each zero is counted with multiplicity. 
\end{lemma}
\begin{proof}
Follows from taking $\zeta(s)$ and $s = 1 + it$ in Ford \cite[Lemma~2.2]{Ford2000}.
\end{proof}

As is common practice, instead of working directly with Lemma~\ref{fordlem22}, we consider a ``mollified" version, presented in Lemma~\ref{fordlem46}. The choice of the smoothing function $f$ significantly influences the eventual zero-free region constant. As in Ford \cite{Ford2000}, we base our choice of $f$ on Lemma~7.5 of Heath-Brown \cite{heath_brown_zero_1992}. Jang and Kwon \cite{jang_note_2014} obtained improvements for the classical zero-free region by choosing a different mollifier described in Xylouris \cite{Xylouris}, however we find that such a choice of $f$ did not produce significant improvements here for the Korobov--Vinogradov zero-free region.  

We construct the smoothing function the same way as Ford \cite{Ford2000}, which we briefly review here for completeness. 
Given a qualifying trigonometric polynomial $P(x)=\sum_{k=0}^K b_k \cos(kx)$, let $\theta = \theta(b_0,b_1)$ be the unique solution to the equation 
\begin{equation}\label{theta_defn}
\sin^2\theta = \frac{b_1}{b_0}\left(1 - \theta\cot \theta\right),\qquad 0 < \theta < \frac{\pi}{2}.
\end{equation}
For the choice of $b_0, b_1$ in \eqref{eqnPoly40}, we compute
\begin{equation}\label{eqnTheta40}
\theta = 1.13331020636698\ldots
\end{equation}and for \eqref{eqnPoly46} we find
\[
\theta = 1.13269369969232\ldots\,.
\]

Define
\begin{gather*}
g(u) := \begin{cases}
(\cos(u\tan \theta) - \cos\theta)\sec^2\theta, &\abs{u} \le \theta\cot\theta,\\
0, &\text{otherwise},
\end{cases}\\
w(u) := (g * g)(u) = \int_{-\infty}^{\infty}g(t)g(u - t)\,\text{d}t,
\end{gather*}
where, in particular,
\begin{equation}\label{w0_defn}
w(0) = (\theta\tan\theta + 3\theta\cot\theta - 3)\sec^2\theta\,.
\end{equation}
We then choose the smoothing function to be 
\begin{equation}\label{f_defn}
f(u) := \lambda e^{\lambda u}w(\lambda u)
\end{equation}
where $\lambda$ is a positive parameter to be fixed later. Note that $f(u) \ge 0$ since $g(u) \ge 0$. 

We will primarily require properties about the Laplace transform of this smoothing function. Let
\begin{equation*}
F(z) := \int_{0}^{\infty}e^{-zu}f(u)\,\text{d}u
\end{equation*}
denote the Laplace transform of $f(u)$, and similarly let $W(z)$ be the Laplace transform of $w(u)$. The function $W(z)$ has a closed formula, given in Ford \cite{Ford2000}, which we state here for convenience:
\begin{equation*}
W(z) = \frac{w(0)}{z} + W_0(z),
\end{equation*}
\begin{equation}\label{W0_defn}
W_0(z) := \frac{c_0\left(c_2\left[(z + 1)^2e^{-2\theta(\cot\theta)z} + z^2 - 1\right] - c_1z - c_3z^3\right)}{z^2(z^2 + \tan^2\theta)^2},
\end{equation}
where
\begin{align*}
c_0 &:= \frac{1}{\sin\theta \cos^3\theta},
&c_1 := (\theta - \sin\theta \cos\theta)\tan^4\theta,\\
c_2 &:= \tan^3\theta\sin^2\theta,
&c_3 := (\theta - \sin\theta\cos\theta)\tan^2\theta.
\end{align*}
In particular, via a direct substitution, we have 
\begin{equation}\label{Wp0_formula}
W'(0) = \frac{\csc\theta(3(4\theta^2 - 5) + \theta(15 - 4\theta^2)\cot\theta) - 3\theta\sec\theta}{3\sin\theta}.
\end{equation}
This will be useful later in the proof of Lemma \ref{real_part_inequality_intermediate}. Meanwhile, for $R \ge 3$, and using \eqref{W0_defn}, we have
\begin{equation*}
\abs{W_0(z)} \le \frac{H(R)}{\abs{z}^3},\quad \Re z \ge -1,\; \abs{z} \ge R,
\end{equation*}
where 
\begin{equation*}
H(R) := \frac{c_0}{\left(1 - \frac{\tan^2\theta}{R^2}\right)^2}\left\{c_2\frac{(R + 1)^2}{R^3}\left(e^{2\theta\cot\theta} + 1\right) + \frac{c_1}{R^2} + c_3\right\}.
\end{equation*}
This allows us to bound $F_0(z) := F(z) - f(0) / z$, via the identity
\begin{equation*}
F(z) = W\left(\frac{z}{\lambda} - 1\right),
\end{equation*}
which is a consequence of \eqref{f_defn}. We obtain, as per Ford, 
\begin{equation*}
\abs{F_0(z)} \le C_5(R)\frac{\lambda f(0)}{\abs{z}^2}, \quad \Re z \ge 0,\; \abs{z} \ge (R + 1)\lambda,
\end{equation*}
where 
\begin{equation}\label{C5_defn}
C_5(R) := \frac{H(R)(R + 1)^2}{R^3w(0)} + 1 + \frac{1}{R}.
\end{equation}
The motivation for bounding $\abs{F_0(z)}$ is shown in the next lemma. 
   
\begin{lemma}\label{fordlem46}Let $f$ be a non-negative, compactly supported real function with a continuous derivative and an absolutely convergent Laplace transform $F(z)$ for $\Re z > 0$, where 
\[F(z) := \int_0^{\infty}f(t)e^{-zt}\,\text{d}t,\]
and write $F_0(z) := F(z) - \frac{f(0)}{z}$. Suppose further that for $0 < \eta \le 3/2$ we have
\[\abs{F_0(z)} \le \frac{D}{\abs{z}^2}\quad \forall\; \Re z \ge 0, \abs{z} \ge \eta,\]
for some absolute constant $D$. If $s = 1 + it$ with $t \ge 1000$, then
\begin{equation}\label{fordlem46eqn}
\begin{split}
&\Re \sum_{n\ge 1}\frac{\Lambda(n)f(\log n)}{n^{s}} \le -\!\!\!\!\!\!\sum_{\abs{1 + it - \rho} \le \eta}\Re\left\{F(s - \rho) + f(0)\left(\frac{\pi}{2\eta}\cot\left(\frac{\pi(s - \rho)}{2\eta}\right) - \frac{1}{s - \rho}\right)\right\}\\
&\qquad\qquad + \frac{f(0)}{4\eta}\left\{\int_{-\infty}^{\infty}\frac{\log \abs{\zeta(s - \eta + \frac{2\eta ui}{\pi})}}{\cosh^2 u}\,\text{d}u - \int_{-\infty}^{\infty}\frac{\log \abs{\zeta(s + \eta + \frac{2\eta ui}{\pi})}}{\cosh^2 u}\,\text{d}u\right\}\\
&\qquad\qquad + D\left\{1.8 + \frac{\log t}{3} + \sum_{\abs{s - \rho} \ge \eta}\frac{1}{\abs{s - \rho}^2}\right\}
\end{split}
\end{equation}
and 
\begin{equation*}
\sum_{n\ge 1}\frac{\Lambda(n)f(\log n)}{n} \le F(0) + 1.8D.
\end{equation*}
\end{lemma}
\begin{proof}
This follows by combining \cite[Lemma~4.5]{Ford2000} with \cite[Lemma~4.6]{Ford2000}. 
\end{proof}
We remark that some negligible improvements are possible if one takes $t\geq t_{0}$ for some $t_{0}> 1000$. These do not affect our final results, given the number of decimal places to which they are stated.

We seek to determine an upper bound on linear combinations of the right side of \eqref{fordlem46eqn}.
We briefly outline our approach, which follows Ford \cite{Ford2000}, while incorporating some improvements from Section~\ref{secPrelims}.
The first integral,
\begin{equation*}
\int_{-\infty}^{\infty}\frac{\log \abs{\zeta(s - \eta + \frac{2\eta ui}{\pi})}}{\cosh^2 u}\,\text{d}u,
\end{equation*}
is taken on a vertical line inside the critical strip, and it can be bounded using Lemma~\ref{fordlem34} below and \eqref{rick}.
This term is by far the most significant, and highlights the sensitivity of the resulting zero-free region to the constants $A$ and $B$ appearing in \eqref{rick}.
The contour of the second integral lies outside the critical strip, and for this we employ Ford's trick of combining $\log\abs{\zeta(\cdot)}$ terms, combined with \eqref{sigma_bound_1}.
This term is the subject of Lemma~\ref{fordlem51}.
Next, the sum
\begin{equation*}
\sum_{\abs{s - \rho} \ge \eta}\frac{1}{\abs{s - \rho}^2}
\end{equation*}
is bounded with the aid of $N(t, v)$, the number of zeros $\rho$ with $\abs{1+it - \rho} \le v$.
This is discussed in Lemmas~\ref{fordlem42} and~\ref{fordlem43}. In Lemma~\ref{fordlem71}, we combine all these results with the trigonometric polynomial \eqref{eqnPoly40} to establish an inequality involving the real and imaginary parts of a zero. 

\begin{lemma}[Ford \cite{Ford2000}]\label{fordlem34}
Suppose that, for fixed $\frac{1}{2} \le \sigma < 1$ and $t \ge 3$, we have
\[
\abs{\zeta(\sigma + i\tau)} \le Xt^Y\log^Zt, \quad \forall \; 1 \le \abs{\tau} \le t,
\]
where $X, Y, Z$ are positive constants and $Y + Z > 0.1$. If $0 < a \le \frac{1}{2}$, $t \ge 100$ and $\frac{1}{2} \le \sigma \le 1 - t^{-1}$, then
\[
\frac{1}{2}\int_{-\infty}^{\infty}\frac{\log\abs{\zeta(\sigma + it + iau)}}{\cosh^2 u}\,\text{d}u \le \log X + Y\log t + Z\log\log t.
\]
\end{lemma}
\begin{proof}
See \cite[Lemma~3.4]{Ford2000}.
\end{proof}

\begin{lemma}\label{fordlem51}
Let $K > 1$ and $b_j$ be the coefficients of a non-negative trigonometric polynomial of degree $K$. Furthermore let $t_1, t_2\in \mathbb{R}$ and $\eta > 0$. Then 
\[
\int_{-\infty}^{\infty}\frac{1}{\cosh^2 u}\sum_{j = 1}^K b_j \log\abs{\zeta(1 + \eta + ijt_1 + iut_2)}\,\text{d}u \ge -2b_0\log\zeta(1 + \eta).
\]
\end{lemma}
\begin{proof}
This is an immediate generalization of \cite[Lemma~5.1]{Ford2000} to degree $K$ polynomials. 
\end{proof}

\begin{lemma}\label{fordlem42}
Let $0 < \eta \le 1/4$ and $t \ge 100$. Then
\begin{equation*}
N(t, \eta) \le 1.3478\eta^{3/2}B\log t + 0.479 + \frac{\log A - \log \eta + \frac{2}{3}\log\log t}{1.879}.
\end{equation*}
\end{lemma}
\begin{proof}
Same as in \cite[Lemma~4.2]{Ford2000}.\footnote{We take this opportunity to correct a minor oversight in the corresponding statement in \cite{Ford2000}.
There, in our notation the hypothesis required $0 < \eta \le 1/4$, and then the inequality
\begin{equation}\label{ford_zeta_sigma_bound}
\zeta(\sigma) \le 0.6 + \frac{1}{\sigma - 1},\quad 1 < \sigma \le 1.06
\end{equation}
from \cite[Lemma~3.1]{Ford2000} was used with $\sigma = 1 + 3.1421\eta$, so here one should require $0<\eta \le 0.019\ldots$ instead.
This does not affect subsequent work: this lemma was employed in the proof of \cite[Thm.~2]{Ford} where $\eta$ was restricted by $0< \eta \le 0.01$.
In our treatment, we retain the condition $0 < \eta \le 1/4$, and opt to achieve a result of similar strength by using \eqref{sigma_bound_1} in place of \eqref{ford_zeta_sigma_bound}.}
Using $\eta$ in place of $R$ in Ford's treatment, we replace the second inequality in (4.1) of \cite{Ford2000} with 
\begin{equation*}
\frac{1}{\abs{\zeta(s + 2.5\eta + iv)}} \le \zeta(1 + 3.1421\eta) \le \frac{e^{3.1421\gamma \eta}}{3.1421 \eta}
\end{equation*}
for real $v$, where the last inequality follows from \eqref{sigma_bound_1}. Hence, since $\eta \le 1/4$,
\begin{equation*}
\log \left(\frac{e^{3.1421\gamma \eta}}{3.1421\eta}\right) = 3.1421\gamma \eta - \log \eta - \log 3.1421 \le -\log \eta - 0.6914.
\end{equation*}
The constant to replace the $0.49$ appearing in \cite[Lemma~4.2]{Ford2000} is then
\begin{equation}\label{eqnNTeta}
\frac{1}{0.3758}\left(\frac{1}{3.1421} - \frac{0.6914}{5}\right) \le 0.479,
\end{equation}
and the result immediately follows from this.
\end{proof}

\begin{lemma}\label{fordlem43}
Suppose $t\ge 10000$ and $0 < \eta \le 1/4$. Then
\begin{align*}
\sum_{\abs{1 + it - \rho} \ge \eta}\frac{1}{\abs{1 + it - \rho}^2} &\le \left[5.409 + 5.392B\left(\frac{1}{\sqrt{\eta}} - 2\right)\right]\log t + 206.7\\
&\qquad+ \frac{1}{\eta^2}\left\{\frac{\log A - \log \eta + \frac{2}{3}\log\log t}{1.879} + 0.213 - N(t, \eta)\right\}.
\end{align*}
\end{lemma}
\begin{proof}
We proceed in the same manner as \cite[Lemma~4.3]{Ford2000}, except we use \eqref{NT_bound} in place of a classical inequality of Rosser \cite{Rosser41}.
In addition, we propagate the new constant \eqref{eqnNTeta} in the bound for $N(t, \eta)$ through to this bound, where we obtain 
\begin{equation*}
0.479 - \frac{1}{2(1.879)} \le 0.213.\qedhere
\end{equation*}
\end{proof}

Equipped with these lemmas, we are now prepared to form an inequality involving the real and imaginary parts of a zero $\rho = \beta + it$. 
\begin{lemma}\label{fordlem71}
Let $0 < \eta \le 1/4$ and $R \ge 3$ be constants.  
Suppose $\beta + it$ is a zero satisfying $t \ge 10000$ and $1 - \beta \le \eta / 2$.
Further, suppose that there are no zeros in the rectangle 
\[
1 - \lambda < \Re s \le 1,\quad t - 1 \le \Im s \le Kt + 1,
\]
where $\lambda$ is a constant satisfying $0 < \lambda \le \min\left\{1 - \beta , \eta / (R + 1)\right\}$. If $b_0$, $b_1$ and $b$ are constants associated with a degree $K$ non-negative trigonometric polynomial, then
\begin{align*}
&\frac{1}{\lambda}\left(\cos^2\theta - \frac{W'(0)b_1}{w(0)b_0}\left(\frac{1 - \beta}{\lambda} - 1\right)\right) \le 0.087\pi^2\frac{b_1}{b_0}\frac{1 - \beta}{\eta^2} \\
&\qquad + \frac{1}{2\eta}\left\{\frac{b}{b_0}\left(\frac{2}{3}L_2 + B\eta^{3/2}L_1 + \log A\right) + \log\zeta(1 + \eta) \right\}\\
&\qquad + C_5(R)\frac{b}{b_0}\lambda\bigg\{\frac{L_1}{3} + \left[5.409 + 5.392B\left(\frac{1}{\sqrt{\eta}} - 2\right)\right]L_1 + 209.1\\
&\qquad\qquad\qquad\qquad + \frac{1}{\eta^2}\left[\frac{\log(A / \eta) + \frac{2}{3}L_2}{1.879} + 0.213\right]\bigg\},
\end{align*}
with $C_5(R)$, $w(0)$ and $W'(0)$ are defined in \eqref{C5_defn},  \eqref{w0_defn} and \eqref{Wp0_formula} respectively.
\end{lemma}
\begin{proof}
We follow the argument of \cite[Lemma~7.1]{Ford2000}, using our Lemma~\ref{fordlem43} instead of \cite[Lemma~4.3]{Ford2000}, using Lemma~\ref{fordlem51} in place of \cite[Lemma~5.1]{Ford2000}, and using the trigonometric polynomial \eqref{eqnPoly40} instead of \eqref{ford_poly_coefficients}.
\end{proof}

To form an explicit zero-free region for $\zeta(s)$, we apply Lemma \ref{fordlem71} with $\eta$ as a fixed function of $t$.
The rate at which $\eta\rightarrow 0$  as $t \to \infty$ determines the shape and width of the zero-free region.
We choose, as in \cite{Ford2000},
\begin{equation}\label{eta_defn}
\eta(t) := E\left(\frac{L_2}{BL_1}\right)^{2/3}
\end{equation}
for some constant $E > 0$. Ultimately, we take
\begin{equation*}
E := 1.8821259
\end{equation*}
in order to attain the constant of $55.241$ appearing in Theorem~\ref{Main}.
This choice replaces that of 
\begin{equation}\label{ford_E_defn}
E = \left(\frac{4}{3}\left(1 + \frac{b_0}{b}\right)\right)^{2/3},
\end{equation}
appearing in Ford \cite{Ford2000}, where $b_0$ and $b$ are constants from the trigonometric polynomial in \eqref{eqnPoly40}. Ford's choice minimizes the asymptotic zero-free region constant, whereas we choose $E$ to minimize the zero-free region constant holding for all $t \ge 3$. The difference between our choices reflects the presence of a secondary error term that decreases as $E$ increases, and which is significant for small values of $t$. 

\section{Proof of Theorem~\ref{Main}}\label{sec:mainproof}

We divide our argument into four sections depending on the size of $t$. Throughout, let 
\begin{equation*}
\nu(t) := \frac{1}{55.241(\log t)^{2/3}(\log\log t)^{1/3}}.
\end{equation*}
First, for $3 \le t \le H := 3\cdot 10^{12}$, we use the rigorous verification of RH up to height $H$ in \cite{platt_riemann_2021}.
Next, for $H \le t < \exp(8928)$, the desired result follows immediately from Theorem \ref{thmClassical}, since 
\begin{equation*}
\frac{1}{5.558691\log t} \ge \nu(t),\quad H \le t < \exp(8928).
\end{equation*}
For $\exp(8928) \le t < \exp(52238)$, we use Theorem~\ref{thmIntZFR}, combined with the observation that
\begin{equation*}
\frac{0.05035}{h(t)} - \frac{0.0349}{h^2(t)} > \nu(t)
\end{equation*}
for this range of $t$, where $h(t)$ is defined in Theorem~\ref{thmIntZFR}.
The proof for $t \ge \exp(52238)$ forms the main part of our argument. For convenience, let
\begin{equation*}
T_0 := \exp(52238),\qquad M_1 := 0.048976.
\end{equation*}
Define the function
\begin{equation}\label{Z_defn}
Z(\beta, t) := (1 - \beta)B^{2/3}(\log t)^{2/3}(\log\log t)^{1/3} 
\end{equation}
and let
\begin{equation}\label{M_defn}
M := \inf_{\substack{\zeta(\beta + it) = 0\\t \ge T_0}}Z(\beta, t).
\end{equation}
If $M \ge M_1$, then we are done, since if $\beta + it$ is a zero with $t \ge T_0$, then
\begin{equation*}
M_1 \le M \le (1 - \beta)B^{2/3}(\log t)^{2/3}(\log\log t)^{1/3}
\end{equation*}
and $B^{2/3}/M_1 < 55.241$. Assume for a contradiction that
\begin{equation}\label{M_assumption}
M < M_1.
\end{equation}
We will show that this implies $M \ge M_1$. Under assumption \eqref{M_assumption}, there is a zero $\beta + it$ such that $t \ge T_0$ and for which $Z(\beta, t)$ is arbitrarily close to $M$. In particular, we may take a zero satisfying
\begin{equation}\label{Z_bounds_1}
M \le Z(\beta, t) \le M(1 + \delta),\quad \delta := \min\left\{\frac{10^{-100}}{\log T_0}, \frac{M_1 - M}{2M}\right\}.
\end{equation}
Note in particular that $Z(\beta, t) < M_1$. Next, let
\begin{equation*}
\lambda := \lambda(t) = \frac{M}{B^{2/3}{L_1}^{2/3}{L_2}^{1/3}}.
\end{equation*}
There are no zeros of $\zeta(s)$ in the rectangular region 
\begin{equation}\label{zerofree_rectangle}
1 - \lambda < \Re s \le 1,\quad t - 1 \le \Im s \le Kt + 1.
\end{equation}
This is because if a zero $\beta' + it'$ exists in that region, then $\log t' \le L_1(t')$ and $\log\log t' \le L_2(t')$, so
\begin{equation*}
1 - \beta' < \lambda(t') \le \frac{M}{B^{2/3}(\log t')^{2/3}(\log\log t')^{1/3}},
\end{equation*}
i.e., $Z(\beta', t') < M$ and $t' \ge t - 1 \ge T_0 - 1$.
However, if $t' \in [T_0 - 1, T_0)$ then by Theorem \ref{thmIntZFR},
\begin{equation*}
\begin{split}
1 - \beta' &\ge \frac{0.05035}{h(t')} - \frac{0.0349}{h^2(t')} \ge \frac{0.05035}{h(T_0)} - \frac{0.0349}{h^2(T_0)} \\
&> \frac{M_1}{B^{2/3}(\log(T_0 - 1))^{2/3}(\log\log(T_0 - 1))^{1/3}} \ge \frac{M_1}{B^{2/3}(\log t')^{2/3}(\log\log t')^{1/3}}
\end{split}
\end{equation*}
so that $Z(\beta', t') \ge M_1 > M$, a contradiction. Note that the third inequality is verified by a numerical computation.
On the other hand if $t' \in [T_0, Kt + 1]$, then by \eqref{M_defn} we also arrive at a contradiction.
Thus any zero $\beta + it$ must satisfy
\begin{equation}\label{lambda_constraint_1}
\lambda \le 1 - \beta,
\end{equation}
thereby proving claim \eqref{zerofree_rectangle}.

Next, we choose $\eta = EB^{-2/3}(L_2/L_1)^{2/3}$, as in \eqref{eta_defn}, with $E = 1.8821259$.
This choice gives, for all $t \ge T_0$, 
\begin{equation}\label{lambda_constraint_2}
\lambda \le \frac{M_1}{B^{2/3}L_1^{2/3}L_2^{1/3}} = \eta\cdot \frac{M_1}{EL_2} < \frac{\eta}{417.48},
\end{equation}
where the last inequality follows from $K = 40$ and $t \ge T_0$.
Furthermore, using $Z(\beta, t) < M_1$
we have that for all $t \ge T_0$, 
\begin{equation}\label{beta_constraint_1}
1 - \beta < \frac{M_1}{B^{2/3}(\log t)^{2/3}(\log\log t)^{1/3}} < \frac{EL_2^{2/3}}{417.48B^{2/3}L_1^{2/3}} < \frac{\eta}{2},
\end{equation}
Finally, by substituting the values of $E$, $T_0$ and $K$, we have 
\begin{equation}\label{eta_constraint}
\eta \le \frac{E}{B^{2/3}}\left(\frac{L_2(T_0)}{L_1(T_0)}\right)^{2/3} < \frac{1}{4}.
\end{equation}
Collecting \eqref{lambda_constraint_1}, \eqref{lambda_constraint_2}, \eqref{beta_constraint_1} and \eqref{eta_constraint}, we see that all of the requirements of Lemma~\ref{fordlem71} are met with $R = 416$. Substituting $b_0 = 1$, $b_1 = 1.74600190914994$, $b = 3.56453965437134$, $K = 40$ (from \eqref{eqnPoly40}) into Lemma~\ref{fordlem71} and computing numerically the values of $w(0)$ and $W'(0)$, we obtain  
\begin{equation}\label{zero_inequality}
\frac{1}{\lambda}\left[0.17949 - 0.20466\left(\frac{1 - \beta}{\lambda} - 1\right)\right] \le 1.5T_1 + T_2 + C_5(R)\frac{b}{b_0} T_3,
\end{equation}
where 
\begin{equation*}
\begin{split}
T_1 &:= \frac{1 - \beta}{\eta^2},\\
T_2 &:= \frac{1}{2\eta}\left[\frac{b}{b_0}\left(\frac{2}{3}L_2 + B\eta^{3/2}L_1 + \log A\right) + \log\zeta(1 + \eta)\right], \\
T_3 &:= \lambda \Bigg[\frac{L_1}{3} + \left[5.409 + 5.392B\left(\frac{1}{\sqrt{\eta}} - 2\right)\right]L_1 + 209.1\\
&\qquad\qquad + \frac{1}{\eta^2}\left\{\frac{\log A - \log \eta + \frac{2}{3}L_2}{1.879} + 0.213\right\}\Bigg].
\end{split}
\end{equation*}
We proceed by bounding each of the terms $T_1$, $T_2$ and $T_3$.
First, using \eqref{Z_defn} we find
\begin{equation}\label{T1_bound}
\begin{split}
T_1 &= \frac{Z(\beta, t)}{B^{2/3}(\log t)^{2/3}(\log\log t)^{1/3}}\cdot \frac{B^{4/3}}{E^2}\left(\frac{L_1}{L_2}\right)^{4/3}\\
&\le \frac{M_1 B^{2/3}}{E^2}\cdot \left(\frac{L_1}{L_2}\right)^{2/3}\cdot \left(\frac{L_1}{\log t}\right)^{2/3}\frac{1}{L_2^{2/3}(\log\log t)^{1/3}}\\
&\le \kappa_1 \frac{M_1 B^{2/3}}{E^2}\left(\frac{L_1}{L_2}\right)^{2/3},
\end{split}
\end{equation}
where, since $t \ge T_0$,
\begin{equation*}
\begin{split}
\kappa_1 &= \left(\frac{L_1(T_0)}{L_2(T_0)}\right)^{2/3}\cdot\frac{1}{(\log T_0)^{2/3}(\log \log T_0)^{1/3}}.
\end{split}
\end{equation*}
Next, using \eqref{sigma_bound_1} and \eqref{eta_constraint}, we have
\begin{equation*}
\log\zeta(1 + \eta) \le \gamma \eta - \log\eta \le \frac{2}{3} L_2 + \frac{2}{3}\log \frac{B}{L_2} - \log E + \frac{\kappa_2 E}{B^{2/3}},
\end{equation*}
where 
\begin{equation*}
\kappa_2 := \gamma\left(\frac{L_2(T_0)}{BL_1(T_0)}\right)^{2/3}.
\end{equation*}
Hence
\begin{equation}\label{T2_bound}
\begin{split}
T_2 &\le \frac{1}{2\eta}\left[\frac{b}{b_0}\left(\frac{2}{3}L_2 + B\eta^{3/2}L_1 + \log A\right) + \frac{2}{3}L_2 + \frac{2}{3}\log \frac{B}{L_2} - \log E + \frac{\kappa_2 E}{B^{2/3}}\right]\\
&= \left[\frac{1}{3E}\left(\frac{b}{b_0} + 1\right) + \frac{bE^{1/2}}{2b_0}\right]B^{2/3}L_1^{2/3}L_2^{1/3} \\
&\qquad\qquad + \frac{1}{2E}\left[\frac{b}{b_0}\log A + \frac{2}{3}\log \frac{B}{L_2} - \log E + \frac{\kappa_2 E}{B^{2/3}}\right]B^{2/3}\frac{L_1^{2/3}}{L_2^{2/3}}.
\end{split}
\end{equation}
Next, we have 
\begin{equation}\label{T3_bound}
\begin{split}
T_3 &\le \frac{M_1}{B^{2/3}L_1^{2/3} L_2^{1/3}}\Bigg[\left(\kappa_3 + \frac{5.392B^{4/3}}{E^{1/2}}\left(\frac{L_1}{L_2}\right)^{1/3} - 10.784B\right)L_1 \\
&\qquad\qquad + \frac{1}{E^2}\left(\frac{BL_1}{L_2}\right)^{4/3}\left\{\frac{\log A + \frac{4}{3}L_2 + \frac{2}{3}\log\frac{B}{L_2} - \log E}{1.879} + 0.213\right\}\Bigg]\\
&= M_1\left(\frac{B L_1}{L_2}\right)^{2/3}\Bigg[\frac{5.392}{E^{1/2}} + \frac{4}{5.637E^2} + \frac{\kappa_3 - 10.784B}{B^{4/3}}\left(\frac{L_2}{L_1}\right)^{1/3}\\
&\qquad\qquad + \frac{1}{E^2L_2}\left\{\frac{\log A + \frac{2}{3}\log\frac{B}{L_2} - \log E}{1.879} + 0.213\right\}\Bigg],
\end{split}
\end{equation}
where, since $t \ge T_0$ and $L_1(T_0) > \log(KT_0)$, we may take
\begin{equation*}
\kappa_3 = \frac{1}{3} + 5.409 + \frac{209.1}{\log K + \log T_0}.
\end{equation*}
Last, we have
\begin{equation}\label{T4_bound}
\begin{split}
\frac{1 - \beta}{\lambda} - 1 &\le (1 + \delta)\left(\frac{L_1}{\log t}\right)^{2/3}\left(\frac{L_2}{\log\log t}\right)^{1/3} - 1 \\
&< \left(1 + \frac{10^{-100}}{\log T_0}\right)\frac{L_1}{\log t} - 1
\le \frac{\kappa_4}{\log T_0},
\end{split}
\end{equation}
where, since $t \ge T_0$,
\begin{equation*}
\kappa_4 := 10^{-100}\left(1 + \frac{\log(K + T_0^{-1})}{\log T_0}\right) + \log(K + T_0^{-1}).
\end{equation*}
Combining \eqref{zero_inequality}, \eqref{T2_bound}, \eqref{T3_bound}, \eqref{T1_bound} and \eqref{T4_bound}, we conclude
\begingroup
\allowdisplaybreaks
\begin{align*}
\frac{1}{\lambda}&\left(0.17949 - \frac{0.20466\kappa_4}{\log T_0}\right) \le \frac{1}{\lambda}\left[0.17949 - 0.20466\left(\frac{1 - \beta}{\lambda} - 1\right)\right]\notag\\
&\le 1.5\frac{1 - \beta}{\eta^2} + \frac{1}{2\eta}\left[\frac{b}{b_0}\left(\frac{2}{3}L_2 + B\eta^{3/2}L_1 + \log A\right) + \log\zeta(1 + \eta)\right]\\
&\qquad + C_5(R)\frac{b}{b_0}\lambda \left[\frac{L_1}{3} + \left[5.409 + 5.392B\left(\frac{1}{\sqrt{\eta}} - 2\right)\right]L_1 + 209.1\right.\\
&\qquad\qquad \left. +  \frac{1}{\eta^2}\left\{\frac{\log A - \log \eta + \frac{2}{3}L_2}{1.879} + 0.213\right\}\right]\\
&\le \frac{1.5\kappa_1 M_1 B^{2/3}}{E^2}\left(\frac{L_1}{L_2}\right)^{2/3} + \left[\frac{1}{3E}\left(\frac{b}{b_0} + 1\right) + \frac{bE^{1/2}}{2b_0}\right](BL_1)^{2/3}L_2^{1/3}\\
&\qquad\qquad + \frac{1}{2E}\left[\frac{b}{b_0}\log A + \frac{2}{3}\log \frac{B}{L_2} - \log E + \frac{\kappa_2 E}{B^{2/3}}\right]\left(\frac{BL_1}{L_2}\right)^{2/3}\\
&\qquad\qquad + C_5(R)\frac{b}{b_0} M_1\left(\frac{B L_1}{L_2}\right)^{2/3}\left[\frac{5.392}{E^{1/2}} + \frac{4}{5.637E^2} + \frac{\kappa_3 - 10.784B}{B^{4/3}}\left(\frac{L_2}{L_1}\right)^{1/3}\right.\\
&\qquad\qquad\qquad\qquad \left. + \frac{1}{E^2 L_2}\left\{\frac{\log A + \frac{2}{3}\log\frac{B}{L_2} - \log E}{1.879} + 0.213\right\}\right].
\end{align*}
\endgroup
At this point we substitute the values $A = 76.2$, $B = 4.45$, $K = 40$, $b_0 = 1$,
$b = 3.56453965437134$, $M_1 = 0.048976$, $E = 1.8821259$ and $R = 416$, and using \eqref{eqnTheta40} and \eqref{C5_defn} to compute $C_5(R)$, we obtain
\begin{equation*}
\frac{1}{\lambda}\left(0.17949 - \frac{0.755}{\log T_0}\right) \le (BL_1)^{2/3}L_2^{1/3}\left[3.25351 + Y(t) \right]
\end{equation*}
where 
\begin{equation*}
Y(t) := \frac{4.940431}{L_2} + \frac{0.136899}{L_2^2} - \frac{1.031863}{L_1^{1/3}L_2^{2/3}} - \frac{0.177104\log L_2}{L_2} - \frac{0.0179076\log L_2}{L_2^2}.
\end{equation*}
A short \textit{Mathematica} computation is used to verify that $Y(t)$ is decreasing for $t \ge T_0$.\footnote{Our monotonicity argument is used to overcome a small issue in Ford's \cite{Ford2000} treatment (after corrections in \cite{Ford2}). In (8.9), the following inequality was used: \[\frac{0.04893}{L_2}\left(\log A + \frac{2}{3}\log \frac{B}{L_2}\right) \le  0.0048\log A + 0.0031\log \frac{B}{L_2},\quad t \ge \exp(30000).\] This should be reversed if $t$ is sufficiently large, as both sides of the inequality are negative.} Therefore, $Y(t) \le Y(T_0) < 0.4110503$.
However, from the definition of $\lambda$ together with $t \ge T_0$, we have
\begin{equation}\label{main_proof_last_line}
\begin{split}
M = \lambda B^{2/3}{L_1}^{2/3}{L_2}^{1/3} &\ge \frac{0.17949 - 0.755/\log T_0}{3.25351 + Y(t)} \\
&\ge \frac{0.17949 - 0.755/52238}{3.25351 + 0.4110503} > 0.04897601 > M_1,
\end{split}
\end{equation}
hence we have arrived at the desired contradiction.

\section{Proof of Theorem~\ref{thmIntZFR}}\label{sec:intermediate_proof}

The proof of this theorem is similar to that of Theorem~\ref{Main}. Instead of bounding $\zeta(s)$ on $\sigma = 1 \pm \eta(t)$, we use upper bounds on $\zeta(s)$ on the lines $\sigma = 1/2$ and $\sigma = 3/2$, i.e., $\eta$ is fixed at $1/2$. The main disadvantage of this new scheme is that the best known bounds of $\zeta(1/2 + it)$ are of order $t^{\theta + \epsilon}$ for some fixed $\theta > 0$, which means that the resulting zero-free region will only have width $O(1/\log t)$. Nevertheless, the resulting zero-free region has a better asymptotic constant than \eqref{eqnMTc}, so we use this result to cover the range $\exp(8928) \le t \le \exp(52238)$. 

Throughout this section, let us define 
\begin{equation}\label{J_defn}
J(t) := \frac{27}{164}\log t + \log 307.098.
\end{equation}
This is the logarithm of the sub-Weyl bound in \eqref{subweyl_bound}. First, we require the following lemma, which corresponds to Lemma~\ref{fordlem71} for this zero-free region.

\begin{lemma}\label{real_part_inequality_intermediate}
Suppose $\zeta(\beta + it) = 0$ where $t \ge \exp(1000)$ and $\beta \ge 1 - 1712^{-1}$. Let $b$, $b_0$ and $K$ be defined in \eqref{eqnPoly40}. Suppose there exists $\lambda \in (0, 1 - \beta]$ such that there are no zeros in the rectangle  
\[
1 - \lambda < \Re s \le 1,\quad t - 1 \le \Im s \le Kt + 1.
\]
Then
\begin{align*}
\frac{1}{\lambda}\left(0.17949 - 0.20466\left(\frac{1 - \beta}{\lambda} - 1\right)\right) &\le 5.746(1 - \beta) + \frac{b}{b_0}J(Kt + 1) + 0.851 \\
&+ 1.0146\lambda\frac{b}{b_0}(3.5691 L_1 + 5.316L_2 + 18.439),
\end{align*}
where $J(t)$ is defined in \eqref{J_defn}. 
\end{lemma}
\begin{proof}
We largely follow the argument of \cite[Lemma~9.2]{Ford2000}. The main changes are to replace Ford's definition of $J(t)$ with \eqref{J_defn}, use a better trigonometric polynomial, increase the height from which the zero-free region takes effect (to improve its constants) and more carefully bound some secondary terms that arise. 

Let $R := \frac{1}{2(1 - \beta)} - 1 \ge 855$, so that $\eta \le 856\lambda$. Using the trigonometric polynomial \eqref{eqnPoly40}, and with $C_5(R)$ as defined in \eqref{C5_defn}, we have $C_5(R) \le 1.0146$. Suppose that $\abs{z} \ge (R + 1)\lambda$. By the lemma's assumptions, this implies $\abs{z} \ge 1/2$ and hence, as in Section~\ref{sec:aux_lemmas}, 
\begin{equation*}
\abs{F_0(z)} \le C_5(R)\frac{\lambda f(0)}{\abs{z}^2} \le \frac{1.0146\lambda f(0)}{\abs{z}^2},\quad \Re z \ge 0, \abs{z} \ge (R + 1)\lambda.
\end{equation*}
Therefore the conditions of Lemma~\ref{fordlem46} are satisfied with $D = 1.0146\lambda f(0)$ and $\eta = 1/2$.
Hence for $\Re s = 1$ and $\Im s \ge 1000$ we have
\begin{equation}\label{classic_ineq}
\begin{split}
&\Re \sum_{n\ge 1}\frac{\Lambda(n)f(\log n)}{n^{s}} \le -\sum_{\abs{s - \rho} \le \frac{1}{2}}\Re\left\{F(s - \rho) + f(0)\left(\pi\cot\left(\pi(s - \rho)\right) - \frac{1}{s - \rho}\right)\right\}\\
&\qquad + \frac{f(0)}{2}\int_{-\infty}^{\infty}\frac{\log \abs{\zeta(s - \frac{1}{2} + \frac{ui}{\pi})} - \log \abs{\zeta(s + \frac{1}{2} + \frac{ui}{\pi})}}{\cosh^2 u}\,\text{d}u\\
&\qquad + D\left\{1.8 + \frac{\log t}{3} + \sum_{\abs{s - \rho} \ge \frac{1}{2}}\frac{1}{\abs{s - \rho}^2}\right\}
\end{split}
\end{equation}
and 
\begin{equation*}
\sum_{n = 1}^\infty\frac{\Lambda(n)f(\log n)}{n} \le F(0) + 1.8D. 
\end{equation*}
We sum \eqref{classic_ineq} for $s = 1 + ijt$ for $0\le j \le K$, and using $\log(jt) \le L_1$ and the sub-Weyl bound \eqref{subweyl_bound} on the line $\sigma = 1/2$, along with the fact that $J(t)$ is increasing, we find
\begin{align}
0 &\le \sum_{n = 1}^\infty\frac{\Lambda(n)}{n}f(\log n)\left[\sum_{j = 0}^Kb_j\cos(jt\log n)\right] = \Re\sum_{j = 0}^{K}b_j\sum_{n = 1}^{\infty}\frac{\Lambda(n)}{n^{1 + ijt}}f(\log n)\notag\\
&\le -\Re\sum_{j = 0}^Kb_j\sum_{\abs{1 + ijt - \rho} \le \frac{1}{2}}\left\{F(1 + ijt - \rho) + f(0)\left(\pi\cot(\pi(1 + ijt - \rho)) - \frac{1}{1 + ijt - \rho}\right)\right\}\notag\\
&\qquad\qquad + bf(0)J(Kt + 1) - \frac{1}{2}\sum_{j = 1}^Kb_j\int_{-\infty}^{\infty}\frac{\log \abs{\zeta(\frac{3}{2} + ijt + \frac{ui}{\pi})}}{\cosh^2 u}\,\text{d}u + b_0 F(0)\notag\\
&\qquad\qquad + D\left\{b\left(1.8 + \frac{L_1}{3}\right) + 1.8b_0 + \sum_{j = 1}^Kb_j\sum_{\abs{1 + ijt - \rho} \ge \frac{1}{2}}\frac{1}{\abs{1 + ijt - \rho}^2}\right\}.\label{classic_bound0}
\end{align}
We now seek an upper bound for the right side of \eqref{classic_bound0}, following largely the same arguments as Ford.
First, as in the proof of \cite[Thm.~3]{Ford2000}, we have
\begin{equation*}
-\frac{1}{2}\int_{-\infty}^{\infty}\frac{1}{\cosh^2 u}\sum_{j = 1}^K b_j \log\left|\zeta\left(\frac{3}{2} + ijt + \frac{iu}{\pi}\right)\right|\,\text{d}u\le 0.851b_0.
\end{equation*}
Note that this bound relies on special properties of $\zeta(s)$ on the line $\sigma = 3/2$, and is sharper than what Lemma~\ref{fordlem51} implies.
Next, by \cite[eq.\ (9.2)]{Ford2000}, we have
\begin{equation}\label{classic_bound1}
\begin{split}
&\sum_{j = 1}^Kb_j\sum_{\abs{1 + ijt - \rho} \ge \frac{1}{2}}\frac{1}{\abs{1 + ijt - \rho}^2}\\
&\qquad\le \sum_{j = 1}^K b_j\left[3.2357\log(jt) + 5.316\log\log(jt) + 16.134 - 4N\left(jt, \frac{1}{2}\right)\right]\\
&\qquad\le b\left[3.2357L_1 + 5.316L_2 + 16.134\right] - 4\sum_{j = 1}^Kb_jN\left(jt, \frac{1}{2}\right).
\end{split}
\end{equation}
In \cite[Lemma~7.1]{Ford2000} it is established, via the maximum modulus principle, that if $\Re z \ge \pi \lambda$ and $\abs{1 + ijt - \rho} \le \frac{1}{2}$, then
\begin{equation}\label{classic_bound2}
\begin{split}
&-\Re\left\{F(1 + ijt - \rho) + f(0)\left(\pi\cot(\pi(1 + ijt - \rho)) - \frac{1}{1 + ijt - \rho}\right)\right\} \le c \pi^2 \lambda f(0),
\end{split}
\end{equation}
where 
\begin{equation*}
c := \frac{4}{\pi^2} + \frac{\pi(1 - \beta)H(R)}{w(0)(\frac{\pi}{2} - \pi(1 - \beta))^2}.
\end{equation*}
We use \eqref{classic_bound2} for $j \ne 1$. If $j = 1$, then $1 + ijt - \rho = 1 - \beta$ and $\cot x - x^{-1} \ge -0.3334x$ for $0 < x \le \pi(1 - \beta)$, so
\begin{equation}
\begin{split}\label{classic_bound4}
&-\Re\left\{F(1 +it - \rho) + f(0)\left(\pi \cot(\pi(1 + it - \rho)) - \frac{1}{1 + it - \rho}\right)\right\} \\
&\qquad\qquad\qquad\qquad \le F(1 - \beta) - 0.3334\pi^2 f(0)(1 - \beta).
\end{split}
\end{equation}
Combining \eqref{classic_bound2} and \eqref{classic_bound4}, and noting that $c\pi^2\lambda f(0)b_1 N(t, \frac{1}{2}) \ge 0$, we find
\begin{equation}\label{classic_bound3}
\begin{split}
&-\Re\sum_{j = 0}^Kb_j\sum_{\abs{1 + ijt - \rho} \le \frac{1}{2}}\left\{F(1 + ijt - \rho) + f(0)\left(\pi\cot(\pi(1 + ijt - \rho)) - \frac{1}{1 + ijt - \rho}\right)\right\}\\
&\qquad\qquad \le -b_1\left[F(1 - \beta) - 0.3334 \pi^2 f(0)(1 - \beta)\right] + c\pi^2\lambda f(0)\sum_{j = 0}^K b_j N\left(jt, \frac{1}{2}\right).
\end{split}
\end{equation}
Substituting \eqref{classic_bound1}, \eqref{classic_bound2} and \eqref{classic_bound3} into \eqref{classic_bound0}, we obtain
\begin{equation}\label{classic_bound5}
\begin{split}
0 &\le -b_1\left[F(1 - \beta) - 0.3334\pi^2 f(0)(1 - \beta)\right] + \left(c\pi^2\lambda f(0) - 4D\right)\sum_{j = 1}^Kb_j N\left(jt,\frac{1}{2}\right) \\
&\qquad\qquad + f(0)\left[bJ(Kt + 1) + 0.851b_0 \right] + b_0 F(0) \\
&\qquad\qquad + D\left\{b\left(1.8 + \frac{L_1}{3} + 3.2357L_1 + 5.316L_2 + 16.134\right) + 1.8b_0\right\}.
\end{split}
\end{equation}
However, from the definition of $D$ and combined with the estimates $H(R) \le 134.87$, $w(0) \ge 5.64531$, we have
\begin{equation}\label{classic_bound8}
c\pi^2\lambda f(0) - 4D = \left[4 + \frac{\pi(1 - \beta)H(R)}{w(0)(\beta - \frac{1}{2})^2} - 4.0584\right]\lambda f(0)\le 0,
\end{equation}
so we may drop the sum in \eqref{classic_bound5}. Furthermore, 
\begin{align}
b_0F(0) - b_1F(1 - \beta) &= b_0W(-1) - b_1W\left(\frac{1 - \beta}{\lambda} - 1\right) \notag\\
&= b_1\left(W(0) - W\left(\frac{1 - \beta}{\lambda} - 1\right)\right)- \left(b_1 W(0) - b_0 W(-1)\right) \notag\\
&= b_1\left(W(0) - W\left(\frac{1 - \beta}{\lambda} - 1\right)\right) - \frac{b_0 f(0)\cos^2\theta}{\lambda}.\label{classic_bound6}
\end{align}
Note that $W(z)$ is decreasing, since
\[
W'(z) = \frac{\text{d}}{\text{d}z}\int_0^{\infty}e^{-zu}w(u)\,\text{d}u = \int_0^{\infty}\frac{\text{d}}{\text{d}z}e^{-zu}w(u)\,\text{d}u = -\int_0^{\infty}ue^{-zu}w(u)\,\text{d}u < 0,
\]
and similarly $W'(z)$ is increasing.
Hence by the mean-value theorem, for some $0 < \xi < (1-\beta)/\lambda - 1$ we have
\begin{equation}\label{classic_bound7}
\begin{split}
W(0) - W\left(\frac{1 - \beta}{\lambda} - 1\right) &= \left(1 - \frac{1 - \beta}{\lambda}\right)W'(\xi)\\
&\le \left(1 - \frac{1 - \beta}{\lambda}\right)W'(0)\\
&\le 0.6617195\left(\frac{1 - \beta}{\lambda} - 1\right),
\end{split}
\end{equation}
where the last inequality follows from \eqref{Wp0_formula}. Substituting \eqref{classic_bound8}, \eqref{classic_bound6} and \eqref{classic_bound7} into \eqref{classic_bound5}, we obtain the desired result. 
\end{proof}

We now proceed to the proof of Theorem~\ref{thmIntZFR}.
Let $J(t)$ be as defined in \eqref{J_defn}, and assume throughout that $t \ge t_0 := \exp(1000)$.
Observe that
\begin{equation*}
J(Kt + 1) - J(t) = \frac{27}{164}\log\left(\frac{Kt + 1}{t}\right) \le 0.60732.
\end{equation*}
We will show that if $\zeta(\beta + it) = 0$ then $1 - \beta > c(t)(0.05035 - 0.0349 c(t))$, where 
\begin{equation*}
c(t) := \frac{1}{J(t) + 1.3686}.
\end{equation*}
Define $Y(\beta, t)$ implicitly via the equation 
\begin{equation}\label{Y_defn}
1 - \beta = \frac{0.05035 - 0.0349 c(t)}{J(t)  + Y(\beta, t)}.
\end{equation}
As $t \to \infty$, we have $Y(\beta, t) \to -\infty$, and $M := \max_{t \ge t_0}Y(\beta, t)$ is well-defined and finite.
If $M \le 1.3686$ then we are done, so assume that $M > 1.3686$, i.e., that there exists a zero $\rho = \beta + it$ such that $Y(\beta, t) = M > 1.3686$.
We pursue a contradiction argument as before.
By our assumption, there is a rectangular region near $\rho$ that is zero-free, specifically $\zeta(s) \ne 0$ for all $s$ satisfying
\begin{equation*}
1 - \lambda < \Re s \le 1\quad\textrm{and}\quad t - 1 \le \Im s \le Kt + 1,
\end{equation*}
with 
\begin{equation*}
\lambda = \frac{0.05035 - 0.0349 c(t)}{J(Kt + 1) + M}.
\end{equation*}
Combining this with the definition of $Y(\beta, t)$, we obtain
\begin{equation*}
\frac{1 - \beta}{\lambda} - 1 = \frac{J(Kt + 1) + M}{J(t) + M} - 1 = \frac{J(Kt + 1) - J(t)}{J(t) + M} \le \frac{0.60732}{J(t) + M}.
\end{equation*}
Applying Lemma~\ref{real_part_inequality_intermediate}, we obtain
\begin{equation*}
\begin{split}
J(Kt + 1) + M &= \frac{0.05035 - 0.0349 c(t)}{\lambda} \\
&\le \frac{b_0}{b}\frac{1}{\lambda}\left(0.17949 - 0.20466\left(\frac{0.60732}{J(t) + M}\right)\right)\\
&\le 5.746\frac{b_0}{b}(1 - \beta) + J(Kt + 1) + 0.851\frac{b_0}{b} \\
&\qquad\qquad\qquad + 1.0146\lambda(3.5691 L_1 + 5.316 L_2 + 18.439),
\end{split}
\end{equation*}
and hence, substituting the values of $b_0$ and $b$, and using \eqref{Y_defn} with $Y(\beta, t) > 1.3686$, we obtain
\begin{align*}
M &\le 1.612(1 - \beta) + 0.23875 + 1.0146\lambda\left(3.5691L_1 + 5.316L_2 + 18.439\right)\\
&\le 0.23875 + c(t)(0.05035 - 0.0349 c(t))(34.384 + 3.6227\log t + 5.3958\log\log t).
\end{align*}
Now for $t \ge t_0 = \exp(1000)$, we have
\begin{equation*}
\begin{split}
&34.384 + 3.6227\log t + 5.3958\log\log t\\
&\qquad\qquad\qquad\le 3.69436\log t \le 22.4399(J(t) + 1.3686) = \frac{22.4399}{c(t)},
\end{split}
\end{equation*}
hence $M \le 0.23875 + 22.4399(0.05035 - 0.0349 c(t)) \le 1.3686$. This achieves our desired contradiction. To complete the proof, we simply note that $J(t) + 1.3686 \le \frac{27}{164}\log t + 7.096$.

\section{Proof of Theorem~\ref{Main2}}\label{sec:proofMain2}

Observe that the contradiction argument in the proof of Theorem \ref{Main} relies on the inequality
\begin{equation}\label{asymp_requirement_M}
M_1 \le \min_{t \ge T_0}\left\{\frac{\cos^2\theta - Y_1(T_0)}{\frac{1}{3E}\left(\frac{b}{b_0} + 1\right) + \frac{bE^{1/2}}{2b_0} + Y(t)}\right\}
\end{equation}
which appears in \eqref{main_proof_last_line}, where $Y_1(t), Y(t) \to 0$ as $t \to \infty$. If we choose $E$ as in \eqref{ford_E_defn}, then \eqref{asymp_requirement_M} is satisfied for any 
\begin{equation*}
M_1 < \cos^2\theta\left(\frac{4}{3}\right)^{2/3}\frac{b_0}{b}\left(1 + \frac{b_0}{b}\right)^{-1/3}
\end{equation*}
so long as we take $T_0$ sufficiently large. Also, as in \cite[\S1]{Ford2000}, by appealing to zero-density theorems, the number of zeros of $\zeta(s)$ in the rectangle $3/4 \le \Re s \le 1$, $T_0 - 1 \le \Im s \le T_0$ is $O(T_0^{1 - \delta})$ for some $\delta > 0$, so for sufficiently large $T_0$, most such rectangles are free of zeros.
Therefore, following the argument in Section~\ref{sec:mainproof}, there are no zeros in the region
\begin{equation*}
\sigma > 1 - \frac{M_1B^{-2/3}}{(\log\abs{t})^{2/3}(\log\log \abs{t})^{1/3}}, \quad \abs{t}\ge T_0
\end{equation*}
for sufficiently large $T_0$.
Using the polynomial $P_{46}(x)$ with values shown in \eqref{eqnPoly46}, we conclude by noting that
\begin{equation}\label{asymp_const_formula}
\cos^2\theta\left(\frac{4}{3}\right)^{2/3}\frac{b_0}{b}\left(1 + \frac{b_0}{b}\right)^{-1/3}B^{-2/3} \ge 0.05617776B^{-2/3} > (48.1588)^{-1}. 
\end{equation}

\begin{remark}
The value of $0.05617776$ improves on $0.05507$ appearing in \cite{Ford2000} and $0.055127$ appearing in \cite{Nielsen}.
\end{remark}

\section{Computations}\label{sec:comp}

The work of Ford \cite{Ford2000} relied in part on the selection of a trigonometric polynomial $P_K(x)$ with certain properties. Recall that 
\[
P_K(x) = \sum_{k = 0}^{K} b_k \cos(kx),
\]
and we require that each $b_k\geq0$, that $b_1>b_0$, and that $P_K(x) \geq 0$ for all real $x$.
Any such polynomial gives rise to an asymptotic zero-free region of the Riemann zeta-function having the form
\[
\sigma \geq 1 - \frac{1}{R_2(\log\abs{t})^{2/3}(\log\log\abs{t})^{1/3}}
\]
when $\abs{t}$ is sufficiently large, and a value for the constant $R_2$ can be computed using the polynomial.
Let $\theta$ be the unique solution in $(0,\pi/2)$ to
\begin{equation}\label{eqnTheta}
b_0 \sin^2\theta = b_1(1-\theta\cot\theta),
\end{equation}
and as in \eqref{eqnPoly40} and \eqref{eqnPoly46} let
\[
b = \sum_{k=1}^K b_k = P_K(0) - b_0.
\]
Using \eqref{asymp_const_formula} we may take
\begin{equation}\label{eqnAsympCons}
R_2 = \frac{1}{\cos^2\theta}\left(\frac{3}{4}\right)^{2/3}\frac{b}{b_0}\left(1 + \frac{b_0}{b}\right)^{1/3}B^{2/3}.
\end{equation}
In \cite{Ford2000} and \cite{Nielsen} respectively, the polynomials \eqref{ford_poly_coefficients} and \eqref{eqnNelsenPoly} with degrees $K=4$ and $5$ were used.
Here we employ a heuristic optimization technique to determine good polynomials $P_K(x)$ with $K > 4$ that produce an improved constant $R_2$.

As in \cite{HoffTrudgian}, we apply the technique of simulated annealing to find favorable trigonometric polynomials $P_K(x)$ having the required properties.
Suppose we have selected a degree $K$.
We must guarantee that any candidate polynomial $P_K(x)$ have the property that $P_K(x) \geq 0$ for all real $x$.
For this, instead of manipulating the coefficients $b_k$ of $P_K(x)$ directly, instead we maintain a list of coefficients $c_0=1$, $c_1$, $c_2$, \ldots, $c_K$.
We set
\[
g(x) = \sum_{k=0}^K c_k e^{i k x}
\]
and let
\[
P_K(x) = \frac{\abs{g(x)}^2}{\sum_{j=0}^K c_j^2}.
\]
Thus, $b_0=1$ and for $k>0$ the coefficient $b_k$ is the $k$th autocorrelation of the sequence of coefficients $c_k$, suitably scaled:
\[
b_k = \frac{2\sum_{j=0}^{K-k} c_j c_{j+k}}{\sum_{j=0}^K c_j^2}.
\]
With this construction, we are assured that $P_K(x) \geq 0$ for all $x$.

Given $K$ and a positive real number $H$, our procedure begins by setting $c_0 = 1$ and selecting each $c_k$ with $1 \le k\leq K$ uniformly at random from the interval $[0,H]$, and then it computes the associated coefficients $b_k$.
Then $P_K(x) \ge 0$ and each $b_k \ge 0$ by construction, and we need only check if $b_1 > b_0$.
If this does not hold we simply pick a new list of values $c_k$ and restart.
We typically find a qualifying polynomial in short order, and we compute its associated value $R_2$ using \eqref{eqnAsympCons}.
We then employ simulated annealing to search for better polynomials from this starting point.

In this process, we maintain a current maximum step size $S$ and temperature $T$.
Given these values, we perform an adjustment to our current polynomial for a certain number of iterations $N$.
In each iteration, we select a positive integer $k < K$ at random, and a random real value $s\in[-S,S]$, then add $s$ to $c_k$ and perform the $O(K)$ operations required to update the $b_k$ values.
If any $b_k < 0$, or if $b_1\leq b_0$, then we reject this adjustment and return $c_k$ to its prior value.
Otherwise, we compute the value $R_2$ for the adjusted polynomial, using Newton's method to determine $\theta$ in \eqref{eqnTheta}.
If the new value is smaller than our prior value, then we keep this adjustment and move to the next iteration.
If the new value is larger than our prior value, then we keep the adjustment with probability depending on the current temperature $T$, otherwise we reject it and return $c_k$ to its prior value.
The threshold probability is $e^{-\Delta R_2/T}$, so that the likelihood of accepting a larger value for $R_2$ is smaller when the change in $R_2$ is greater, but we are more likely to accept larger adjustments when the temperature $T$ is higher.
For a fixed value of $S$, our method executes $N$ iterations for each value of a decreasing sequence of temperature values $T$, culminating with the effective selection $T=0$, so where only improvements to $R_2$ are allowed.
We repeat this for several values of $S$, which decay exponentially.

We used this procedure to search for favorable polynomials with degree from $K = 10$ to $K = 72$.
In each case we typically selected $H\in[100,200]$, step values $S$ decreasing gradually from $50$ or $60$ and slowly decreasing to approximately $2$, using about twelve temperature values $T$, and selecting $N$ near $8000$.
We found many polynomials with $R_2 < 48.18$, and our best polynomial has degree $46$ and is recorded in Table~\ref{tableP46}.
It produces
\[
R_2 = 48.1587921551117,
\]
which we employ in the statement of Theorem~\ref{Main2}.
Figure~\ref{figP40P46}(b) shows a plot of this polynomial over the interval $[\pi/2,\pi]$.

\begin{figure}[tb]
\caption{The trigonometric polynomials employed in the proofs of Theorems~\ref{Main} and~\ref{Main2} respectively, plotted over $[\pi/2,\pi]$.}\label{figP40P46}
\begin{center}
\begin{tabular}{cc}
\includegraphics[width=2.4in]{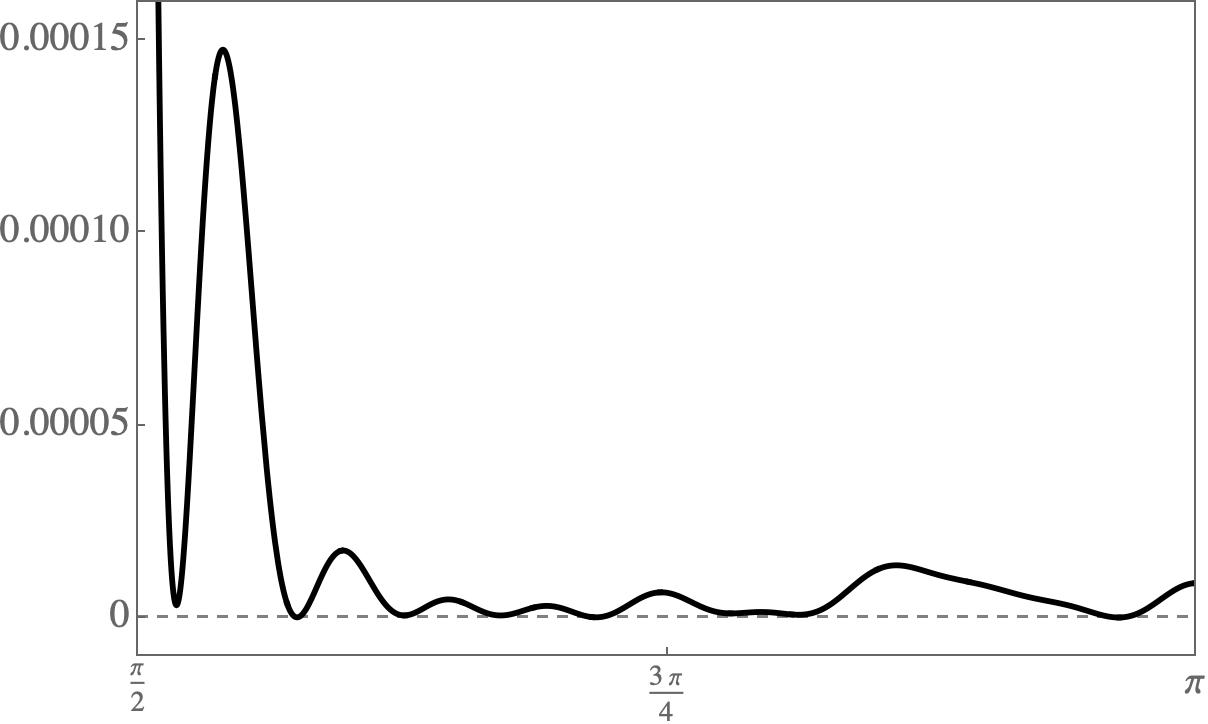} &
\includegraphics[width=2.4in]{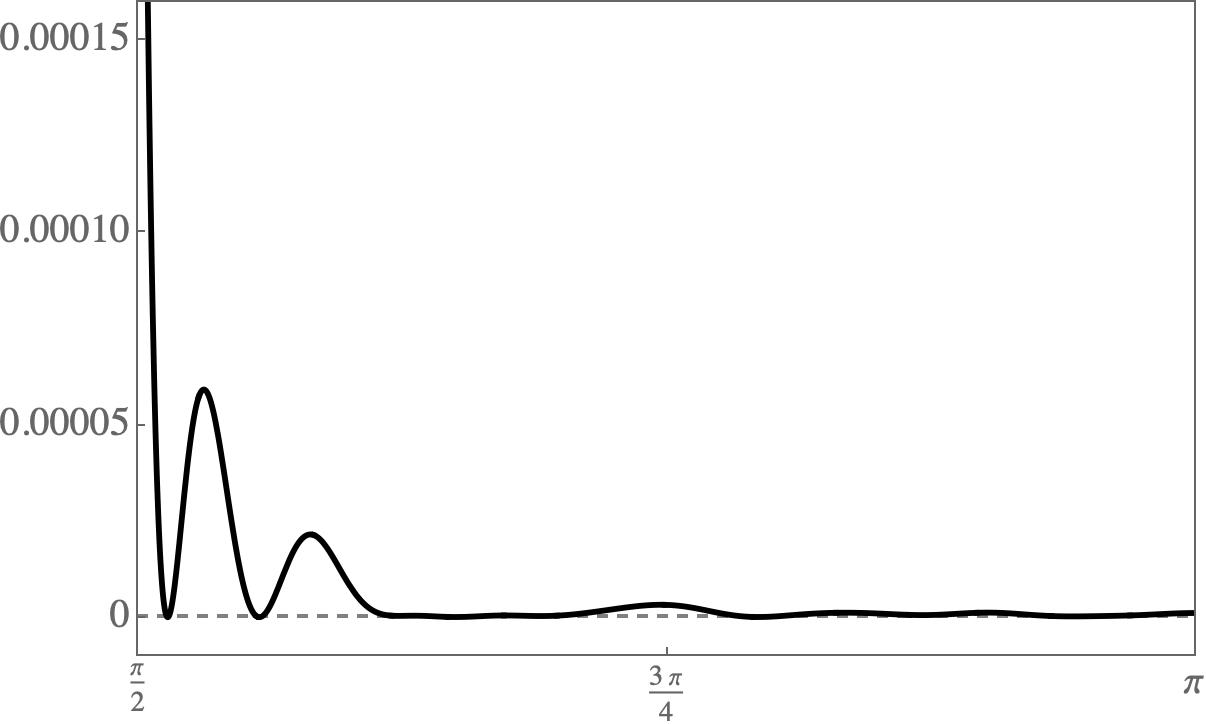}\\
(a) $P_{40}(x)$ &
(b) $P_{46}(x)$
\end{tabular}
\end{center}
\end{figure}

\begin{table}[tbh]
\caption{$P_{46}(x)=\sum_{k=0}^{46}b_k \cos(kx)=\abs{\sum_{k=0}^{46} c_k e^{i kx}}^2/\sum_{k=0}^{46} c_k^2$, with $b=\sum_{k=1}^{46} b_k$.}\label{tableP46}
\tiny
\begin{tabular}{|lll|lll|}\hline
$k$ & $c_k$ & $b_k$ & $k$ & $c_k$ & $b_k$\\\hline
\TS
0 & $1$ & $1$ &
  24 & $49282.888742825$ & $0.000127104592072581$\\
1 & $338.377844758599$ & $1.74708744081848$ &
  25 & $-72469.9665724928$ & $1.74058423843506\cdot10^{-7}$\\
2 & $-219.537480547081$ & $1.14338015090023$ &
  26 & $-80343.7855839228$ & $6.156980223188\cdot10^{-9}$\\
3 & $-736.781312848966$ & $0.521864216745001$ &
  27 & $130557.454262211$ & $7.4923012998548\cdot10^{-5}$\\
4 & $914.902037465737$ & $0.132187571762225$ &
  28 & $456655.665589724$ & $6.29610657045172\cdot10^{-5}$\\
5 & $1915.78694475716$ & $1.44250682908725\cdot10^{-7}$ &
  29 & $686366.255781866$ & $4.51492091998615\cdot10^{-7}$\\
6 & $-1310.28600595906$ & $4.69075278525482\cdot10^{-9}$ &
  30 & $690091.748824027$ & $1.76696516341167\cdot10^{-8}$\\
7 & $-3389.853917904$ & $0.0141904926848435$ &
  31 & $504386.928024044$ & $3.57616762286565\cdot10^{-5}$\\
8 & $1732.46060916218$ & $0.00859097729886965$ &
  32 & $256781.756010027$ & $2.9356535048273\cdot10^{-5}$\\
9 & $6943.01235038993$ & $5.05758761820625\cdot10^{-7}$ &
  33 & $60405.4597040306$ & $2.6547976338407\cdot10^{-7}$\\
10 & $-278.171504957099$ & $4.42284301054098\cdot10^{-10}$ &
  34 & $-37039.4291423529$ & $7.39578841754684\cdot10^{-7}$\\
11 & $-11594.9052445657$ & $0.00262452919575262$ &
  35 & $-49829.9664619879$ & $1.5703528751761\cdot10^{-5}$\\
12 & $-4279.8222109347$ & $0.0018969952017721$ &
  36 & $-22696.5925525196$ & $1.16349907747152\cdot10^{-5}$\\
13 & $14539.7736361703$ & $4.69472495111911\cdot10^{-10}$ &
  37 & $1689.57285600626$ & $1.01423339047177\cdot10^{-7}$\\
14 & $11710.3298598379$ & $2.18058618368512\cdot10^{-7}$ &
  38 & $9780.98700327532$ & $1.71248131672039\cdot10^{-6}$\\
15 & $-18824.0950949349$ & $0.000818384876659817$ &
  39 & $10336.0633101459$ & $7.84636117271159\cdot10^{-6}$\\
16 & $-33323.9900467912$ & $0.000639651965532567$ &
  40 & $9993.04428459519$ & $5.93829512034697\cdot10^{-6}$\\
17 & $-663.769351563045$ & $3.11262094946825\cdot10^{-8}$ &
  41 & $9558.78229646887$ & $9.47232309558493\cdot10^{-7}$\\
18 & $34162.7992046244$ & $7.74994211145798\cdot10^{-7}$ &
  42 & $7861.68784142526$ & $4.84440446543232\cdot10^{-8}$\\
19 & $7425.01374396162$ & $0.000329183630974004$ &
  43 & $6657.72906076572$ & $9.72548049252508\cdot10^{-7}$\\
20 & $-56820.1949038606$ & $0.000268358318561904$ &
  44 & $4736.89926522741$ & $8.45180184576162\cdot10^{-7}$\\
21 & $-60583.1989268389$ & $4.43747297378809\cdot10^{-7}$ &
  45 & $2233.04706685592$ & $2.25111200007826\cdot10^{-7}$\\
22 & $27278.3371854473$ & $1.87358718910571\cdot10^{-7}$ &
  46 & $504.683217557847$ & $6.56678999833493\cdot10^{-10}$\\
23 & $101206.908417904$ & $0.000151428354073652$ & & & $b=3.57440943022073$\\\hline
\end{tabular}
\end{table}

For the result in Theorem~\ref{Main}, we employed two similar procedures.
First, we amended the objective function to compute a value $R_1$ in the zero-free region for all $\abs{t} \ge 3$.
However, each iteration of this computation was much slower, so our computations here were limited.
Indeed, we restricted our searches to degrees $K \le 28$ in this case due to the greater computational complexities.
Second, we used the objective function for the asymptotic constant, and reset our parameters to allow the step size to decrease more rapidly while greatly increasing the number of iterations $N$ per round (taking $N$ between $10^5$ and $10^6$), as well as the number of total number searches performed per degree selection.
This allowed for a much larger search, and we tested degrees $K\leq55$ with this strategy.
Each search recorded polynomials optimized for the asymptotic constant $R_2$, then these polynomials were tested to determine their $R_1$ value for Theorem~\ref{Main}.
One degree-40 polynomial found with this procedure had an asymptotic constant of $R_2 \approx 48.162$, which is inferior to the polynomial displayed in Table~\ref{tableP46}, but it had the best value for $R_1$.
We used this polynomial as our initial state in a further annealing procedure that optimized for $R_1$ to determine an additional small improvement.
Our final polynomial is listed in Table~\ref{tableP40}.
This polynomial $P_{40}(x)$ produced the value $R_1=55.241$ used for Theorem~\ref{Main}.
Figure~\ref{figP40P46}(a) displays a plot of this polynomial over $[\pi/2,\pi]$.

\begin{table}[tbh]
\caption{$P_{40}(x)=\sum_{k=0}^{40}b_k \cos(kx)=\abs{\sum_{k=0}^{40} c_k e^{i kx}}^2/\sum_{k=0}^{40} c_k^2$, with $b=\sum_{k=1}^{40} b_k$.}\label{tableP40}
\tiny
\begin{tabular}{|lll|lll|}\hline
$k$ & $c_k$ & $b_k$ & $k$ & $c_k$ & $b_k$\\\hline
\TS
0 & $1$ & $1$ &
  21 & $14616.1664568754$ & $4.66702819061453\cdot10^{-7}$\\
1 & $8.70590487645377$ & $1.74600190914994$ &
  22 & $15112.6306248979$ & $8.88183754657211\cdot10^{-7}$\\
2 & $253.542513581082$ & $1.14055431833244$ &
  23 & $3281.48150931095$ & $6.61799442215331\cdot10^{-5}$\\
3 & $538.912985014916$ & $0.518966962914028$ &
  24 & $-9858.76392710328$ & $3.70153227317542\cdot10^{-5}$\\
4 & $1421.76588050758$ & $0.130885859164882$ &
  25 & $-11913.1717506499$ & $6.2332255794641\cdot10^{-8}$\\
5 & $3062.1230018832$ & $8.86418531143308\cdot10^{-8}$ &
  26 & $-2607.30174667086$ & $3.29243016002061\cdot10^{-5}$\\
6 & $5755.1498181548$ & $1.79787121328335\cdot10^{-6}$ &
  27 & $6649.42849986177$ & $4.89938220699415\cdot10^{-5}$\\
7 & $9653.05616924715$ & $0.0137716529944408$ &
  28 & $6689.88754688983$ & $1.50988491954013\cdot10^{-5}$\\
8 & $14967.239037407$ & $0.00825900683475376$ &
  29 & $193.678093993709$ & $1.13051732969427\cdot10^{-7}$\\
9 & $21237.398416925$ & $4.91544374578637\cdot10^{-6}$ &
  30 & $-3912.86637215382$ & $2.11823533257304\cdot10^{-5}$\\
10 & $27168.5781338032$ & $2.20263007866541\cdot10^{-6}$ &
  31 & $-2318.83016640653$ & $2.13859401551174\cdot10^{-5}$\\
11 & $31408.8257398599$ & $0.00243120523137902$ &
  32 & $911.79644433382$ & $1.55071932288034\cdot10^{-6}$\\
12 & $32409.0713030987$ & $0.00172926530269636$ &
  33 & $1499.03441911128$ & $1.51812185041036\cdot10^{-6}$\\
13 & $28642.8233658012$ & $1.35500078722447\cdot10^{-6}$ &
  34 & $159.800369623307$ & $1.67615806595912\cdot10^{-5}$\\
14 & $19217.7742754807$ & $2.20879127662495\cdot10^{-6}$ &
  35 & $-551.30680615611$ & $1.60031224178442\cdot10^{-5}$\\
15 & $6084.93971979693$ & $0.00069712400164774$ &
  36 & $-146.185445028008$ & $3.94634065729451\cdot10^{-6}$\\
16 & $-6971.7133423118$ & $0.000530583559753362$ &
  37 & $160.626530894317$ & $4.08859029078879\cdot10^{-7}$\\
17 & $-16051.2777747034$ & $6.3973072524226\cdot10^{-7}$ &
  38 & $9.7531801403406$ & $1.77819241241605\cdot10^{-6}$\\
18 & $-17900.8974008674$ & $5.37323136636712\cdot10^{-7}$ &
  39 & $-46.7104974975636$ & $5.06885733758335\cdot10^{-8}$\\
19 & $-10944.9022767045$ & $0.000234320877800568$ &
  40 & $23.9407317021713$ & $7.50406436813653\cdot10^{-9}$\\
20 & $2745.65474520683$ & $0.000177364641910045$ &
  && $b=3.56453965437134$\\\hline
\end{tabular}
\end{table}

\section{The classical zero-free region}\label{sec:Classical}

In 2005, Kadiri \cite{Kadiri} established the value $R_0 = 5.69693$ in the classical zero-free region of the Riemann zeta-function by means of a clever iterative procedure.
This method relied in part on a particular smoothing function $f(z)$ having certain properties: it was required that $f\in C^2[0,1]$, with compact support, and having a Laplace transform $F(z)$ that is non-negative on the positive real axis.
Kadiri noted that Heath-Brown's work on Linnik's theorem \cite{heath_brown_zero_1992} developed four families of such functions, and that these were well-adapted for application to the problem of the classical zero-free region.
We denote these four families of functions by $f^{(i)}_{\eta,\lambda,\theta}(t)$.
Each is defined by
\[
f^{(i)}_{\eta,\lambda,\theta}(t) = \eta h^{(i)}_{\lambda,\theta}(\eta t)
\]
for $1\leq i\leq4$, where $\lambda>0$, $\theta$, and $\eta$ are real parameters, and $h^{(i)}_{\lambda,\theta}(u)$ are certain functions.
Kadiri employed the fourth of these, where
\begin{equation}\label{eqnh4}
\begin{split}
h^{(4)}_{\lambda,\theta}(u) = \lambda\sec^{2}\theta& \bigg\{\lambda\sec^{2}\theta\left(\frac{-\theta}{\lambda\tan \theta} - \frac{u}{2}\right)\cos(\lambda u\tan\theta) - \frac{2\theta}{\tan \theta} - \lambda u\\
&\quad- \frac{\sin(2\theta + \lambda u\tan \theta)}{\sin2\theta} + 2\left(1 + \frac{\sin(\theta + \lambda u\tan\theta)}{\sin\theta}\right)\bigg\},
\end{split}
\end{equation}
and one requires $\pi/2<\theta<\pi$.
Kadiri set $\lambda=1$ and selected $\theta=1.848$ in her analysis.

In 2014, Jang and Kwon \cite{jang_note_2014} applied all four families of functions $f^{(i)}_{\eta,\lambda,\theta}(u)$ inherited from Heath-Brown's list to this problem, and optimized over $\lambda$ and $\theta$ in each case.
They found that some of Heath-Brown's other functions performed slightly better than \eqref{eqnh4}, and obtained a better value for $R_0$.
Most of their improvement was due to the use of a larger height $T_0$ for which RH had been verified: they used $T_0=3.06\cdot10^{10}$ while Kadiri employed the best value known at the time, approximately $3.3\cdot10^9$.
However, investigation of other auxiliary functions allowed them to reduce their $R_0$ value further.
They found that the four functions $h^{(i)}_{\lambda,\theta}(u)$ produced in turn $5.68372$, $5.68483$, $5.68484$, and $5.68486$ with the new $T_0$ value, so among these their best result arose from $h^{(1)}_{\lambda,\theta}$, which is defined by
\begin{equation}\label{eqnh1}
\begin{split}
h^{(1)}_{\lambda,\theta}(u) = \lambda\sec^{2}\theta& \bigg\{\lambda\sec^{2}\theta\left(\frac{\theta}{\lambda\tan \theta} - \frac{u}{2}\right)\cos(\lambda u\tan\theta) + \frac{2\theta}{\tan \theta} - \lambda u\\
&\quad+ \frac{\sin(2\theta - \lambda u\tan \theta)}{\sin2\theta} - 2\left(1 + \frac{\sin(\theta - \lambda u\tan\theta)}{\sin\theta}\right)\bigg\},
\end{split}
\end{equation}
and one requires $0<\theta<\pi/2$.
Jang and Kwon chose $\lambda=1.03669$ and $\theta=1.13537$, selecting these values in concert with their choice of a non-negative trigonometric polynomial, in order to optimize the constant with this function using Kadiri's method.
As in \cite{Kadiri}, Jang and Kwon selected a favorable non-negative trigonometric polynomial of degree $4$.
Jang and Kwon also employed a function $h^{(5)}$ from Xylouris \cite{Xylouris09}, which out-performed $h^{(1)}_{\lambda,\theta}$ just slightly, producing $R_0=5.68371$.
This was the final value established in \cite{jang_note_2014}.

Independent of \cite{jang_note_2014}, in 2015 the first two authors \cite{HoffTrudgian} determined an improved value for the constant $R_0$ in the classical region by amending Kadiri's method in different ways, and showed that $R_0=5.573412$ is permissible.
A small part of that improvement arose by employing the larger value for $T_0=3.06\cdot10^{10}$.
Most of the gain resulted from two other changes: optimizing over a particular error term, and investigating admissible non-negative trigonometric polynomials of larger degree.
A polynomial of degree $16$ was constructed there by using simulated annealing with an appropriate objective function, and the constant $R_0=5.574312$ was computed using $h^{(4)}_{1,\theta}$ as in \cite{Kadiri}, with an appropriate value of $\theta$.

We take the opportunity here to combine the ideas from \cite{jang_note_2014} and \cite{HoffTrudgian}, together with the recent work \cite{PT} establishing RH to the height $T_0=3\cdot10^{12}$, to record an improved value for the constant $R_0$ in the classical zero-free region of the zeta-function.
We employ the admissible non-negative trigonometric polynomial of degree $16$ from \cite[Table~5]{HoffTrudgian}, the auxiliary function $h^{(1)}_{\lambda,\theta}(u)$ from \eqref{eqnh1}, and the new value for $T_0$.
We set $\lambda=1$ since adjusting this value did produce any further gains of significance, and choose $\theta=1.13489$.
We follow the method detailed in \cite{HoffTrudgian}, with the following adjustments owing to the use of \eqref{eqnh1} rather than \eqref{eqnh4}:
\begin{itemize}
\item Since $\theta$ is now restricted to $(0,\pi/2)$, we set $d_1(\theta) = 2\theta\cot\theta$.
\item We now have $g_1(\theta) = h^{(1)}_{1,\theta}(0) = (\theta\tan\theta+3\theta\cot\theta-3)\sec^2\theta$.
\item We use the inequality $e^y\leq 1+y+y^2/2+y^3/3.47$, which is valid for $d_1(\theta)\leq1.89355$, so for $0\leq\theta\leq 1.13544$.
In \cite{HoffTrudgian}, the value $3.45$ was used in place of $3.47$.
\item For the iteration, we use $r=5$ for our lower bound on the constant we aim to achieve, as in \cite{HoffTrudgian}, but now set the initial upper bound to $R_0=5.573412$.
\item We verify that the error term $C(\eta)$ as defined in \cite[\S2.4]{Kadiri} satisfies $C(\eta)\leq0$ over the interval $[0,1/(5\log(3\cdot10^{12})]$, and that the function $K(w)$ from the same source with $\theta=1.13489$ is increasing on $[0,1]$, as required by the method.
\end{itemize}
We refer the reader to \cite{HoffTrudgian} for full details on the method.
Using this strategy, after seven iterations we compute the value $R_0=5.5586904517$ and establish Theorem~\ref{thmClassical}.
The successive values of $R_0$ determined after each iteration are displayed in Table~\ref{tableItersR03T}, together with the values of other parameters that arise during the calculation---we include these to mirror the data shown in \cite[Table~3]{HoffTrudgian}.

\begin{table}[tb]
\caption{Values of parameters from \cite{HoffTrudgian} in successive iterations when $T_0=3\cdot10^{12}$, using $r=5$ and $\theta=1.13489$.
All values are rounded at the last recorded decimal place.}\label{tableItersR03T}
\begin{tabular}{|cccccc|}\hline
\TS\BS
$R_0$ & $\eta_0\cdot 10^3$ & $\eta_1\cdot10^3$ & $\kappa$ & $\delta$ & $r$\\\hline
$5.5734120$ & $6.259945$ & $0.7565797$ & $0.4410554$ & $0.61994498$ & $5.5603156$\\ 
$5.5603156$ & $6.261572$ & $0.7581420$ & $0.4410415$ & $0.61994923$ & $5.5588702$\\ 
$5.5588702$ & $6.261752$ & $0.7583148$ & $0.4410400$ & $0.61994971$ & $5.5587103$\\ 
$5.5587103$ & $6.261772$ & $0.7583339$ & $0.4410399$ & $0.61994976$ & $5.5586927$\\ 
$5.5586927$ & $6.261774$ & $0.7583360$ & $0.4410398$ & $0.61994976$ & $5.5586907$\\ 
$5.5586907$ & $6.261775$ & $0.7583362$ & $0.4410398$ & $0.61994976$ & $5.5586905$\\ 
$5.5586905$ & $6.261775$ & $0.7583363$ & $0.4410398$ & $0.61994976$ & $5.5586905$\\\hline
\end{tabular}
\end{table}

We remark that with $h^{(4)}_{1,\theta}$ in place of $h^{(1)}_{1,\theta}$ we produced a slightly larger constant with the method, $R_0=5.5608403$, so swapping the smoothing function in this way allowed us to reduce the value of $R_0$ by approximately an additional $2.15\cdot10^{-3}$.
This is about twice the marginal gain that reported in \cite{jang_note_2014} for the same swap of auxiliary functions.
We did not investigate the function of Xylouris, as this had a very small marginal benefit in \cite{jang_note_2014}.

Finally, using this method, and again employing $h^{(1)}_{1,\theta}(u)$ with $\theta=1.13489$ and using the same degree $16$ trigonometric polynomial, we computed the value for $R_0$ that would be achieved by this method if RH were verified up to height $T_0$ in the future, for several values of $T_0 \leq 10^{15}$.
These values are exhibited in Table~\ref{tableClassicalVals}.

\begin{table}[tb]
\begin{center}
\caption{Allowable values for the constant $R_0$ in the classical zero-free region if RH were to be verified to height $T_0$.
All values are rounded up at the last recorded decimal place.}\label{tableClassicalVals}
\begin{tabular}{|cc|}\hline
\TS\BS
$T_0$ & $R_0$\\\hline
\TS
$10^{13}$ & $5.5559836$\\
$3\cdot10^{13}$ & $5.5536904$\\
$10^{14}$ & $5.5513505$\\
$3\cdot10^{14}$ & $5.5493579$\\
$10^{15}$ & $5.5473149$\\\hline
\end{tabular}
\end{center}
\end{table}

\section{Future work}\label{secFuture}

The approach taken in the proof of Theorem \ref{Main} may be summarized as follows. If a zero-free region $\sigma \ge 1 - \nu(t)$ can be established over a small finite region, say for $t \in [T_0 - 1, T_0)$, then under appropriate conditions the same zero-free region holds for all $t \ge T_0$. This suggests an inductive argument may be used---given a sequence of suitable functions $\nu_1(t)$, $\nu_2(t)$, \ldots, $\nu_N(t)$, we may use the zero-free region $\sigma \ge 1 - \nu_j(t)$ to show that there are no zeros in a small finite region, which then implies the next zero-free region $\sigma \ge 1 - \nu_{j + 1}(t)$, and so on. If in addition the first zero-free region can be established unconditionally, this produces a method of iteratively constructing a zero-free region as a union of small zero-free regions.

One possible choice for $\nu_j(t)$ is given by  
\begin{equation*}
\nu_j(t) := \frac{1}{r_j(\log t)^{\phi_j}(\log\log t)^{1 - \phi_j}},\quad t \ge T_0^{(j)},
\end{equation*}
for some $\phi_j > 2/3$, which can be easily established by choosing 
\begin{equation*}
\eta = \frac{E}{B^{2/3}}\left(\frac{L_2}{L_1}\right)^{\phi_j}
\end{equation*}
in place of \eqref{eta_defn}, then following the rest of the proof of Theorem \ref{Main}. If $r_j$ is small enough, then the resulting zero-free region will be sharper than Theorem \ref{Main} over some finite interval $t \in [T_1^{(j)}, T_2^{(j)}]$. By combining multiple such results, we create an envelope of zero-free regions whose union covers the interval $[3, T_0]$ for some large $T_0$. We then use the same argument as Theorem \ref{Main} to cover the range $[T_0, \infty)$. In particular, this allows us to take $T_0$ much larger than is otherwise possible, which reduces the size of $R_1$, the zero-free region constant. 

To attain a non-trivial result via this method, we need to take $r_j$ small enough that $\nu_j(T_0^{(j)}) > \nu_{j - 1}(T_0^{(j)} - 1)$, so there is a small region 
\begin{equation}\label{region_no_zeros}
T_0 - 1 \le t \le T_0, \qquad \nu_{j}(t) \le \sigma \le \nu_{j - 1}(t),
\end{equation}
in which zeros may exist, hence invalidating the inductive argument. Therefore, if we have tools to exclude the possibility of zeros in small, finite regions at known locations within the critical strip, immediate improvements to Theorem \ref{Main} are possible. By judiciously choosing $\phi_j$, we find that using 355 such regions suffice to improve the constant of Theorem \ref{Main} to 52.74, provided that no zeros exist in regions of the form \eqref{region_no_zeros} for each $j$. Conventional arguments, such as raw computation or zero-density estimates, are currently insufficient to completely exclude zeros in these regions due to their large height, ranging from $t \approx \exp(40000)$ to $t \approx \exp(5\cdot 10^7)$. However, if a new method was developed to exclude zeros in small finite regions, then immediate improvements to Theorem \ref{Main} are possible.

\end{document}